\documentclass[10pt]{amsart}
      \usepackage{amsmath,amsfonts}
\def\rg{\hbox to 30pt{\rightarrowfill}}
\def\lg{\hbox to 30pt{\leftarrowfill}}

      \parskip\smallskipamount
          \newtheorem{theorem}{Theorem}[section]
      
      \newtheorem{proposition}[theorem]{Proposition}
      \newtheorem{corollary}[theorem]{Corollary}
      \newtheorem{lemma}[theorem]{Lemma}
      \newtheorem{example}[theorem]{Example}
      
      \newtheorem{remark}[theorem]{Remark}
      \makeatletter
      \@addtoreset{equation}{section}
      \makeatother

      \newcommand{\BB}{{\mathbb B}}
      \newcommand{\CC}{{\mathbb C}}
      \newcommand{\NN}{{\mathbb N}}
      \newcommand{\QQ}{{\mathbb Q}}
      \newcommand{\ZZ}{{\mathbb Z}}
      \newcommand{\DD}{{\mathbb D}}
      
      \newcommand{\FF}{{\mathbb F}}
      \newcommand{\TT}{{\mathbb T}}

      \newcommand{\cA}{{\mathcal A}}
      
      \newcommand{\cC}{{\mathcal C}}
      \newcommand{\cD}{{\mathcal D}}
      \newcommand{\cE}{{\mathcal E}}

      \newcommand{\cH}{{\mathcal H}}
      \newcommand{\cI}{{\mathcal I}}
      \newcommand{\cK}{{\mathcal K}}
      
      \newcommand{\cM}{{\mathcal M}}
      \newcommand{\cN}{{\mathcal N}}
      
      \newcommand{\cQ}{{\mathcal Q}}
      \newcommand{\cP}{{\mathcal P}}
      \newcommand{\cR}{{\mathcal R}}

      \newcommand{\cU}{{\mathcal U}}
      \newcommand{\cV}{{\mathcal V}}

      \newcommand{\cW}{{\mathcal W}}

      \newdimen\expt
      \def\boxit#1{\setbox0\hbox{$\displaystyle{#1}$}
            \hbox{\lower.4\expt
       \hbox{\lower3\expt\hbox{\lower\dp0
            \hbox{\vbox{\hrule height.4\expt
       \hbox{\vrule width.4\expt\hskip3\expt
            \vbox{\vskip3\expt\box0\vskip2\expt}%
       \hskip3\expt\vrule width.4\expt}\hrule height.4\expt}}}}}}
      
\begin{document}
       \pagestyle{myheadings}
      \markboth{ Gelu Popescu}{  Pluriharmonic functions on noncommutative polyballs  }

      \title [   And\^ o   dilations and inequalities on  noncommutative varieties ]
      {      And\^ o   dilations and inequalities on  noncommutative varieties }
        \author{Gelu Popescu}
\date{January 4, 2017}
      \thanks{Research supported in part by  NSF grant DMS 1500922}

      \address{Department of Mathematics, The University of Texas
      at San Antonio \\ San Antonio, TX 78249, USA}
      \email{\tt gelu.popescu@utsa.edu}

\begin{abstract}\
 And\^ o proved a dilation result that implies his celebrated inequality which says that if $T_1$ and $T_2$ are commuting contractions on a Hilbert space, then for any polynomial $p$ in two variables,
 $$
\|p(T_1, T_2)\|\leq \|p\|_{\DD^2},
$$
where $\DD^2$ is the bidisk in $\CC^2$. The main goal of the present paper is to find analogues of And\^ o's results for the elements of the bi-ball ${\bf P}_{n_1,n_2}$ which consists of all pairs $({\bf X}, {\bf Y})$ of row contractions ${\bf X}:=(X_{1},\ldots, X_{n_1})$  and
      ${\bf Y}:=(Y_{1},\ldots, Y_{n_2})$ which commute, i.e. each entry of ${\bf X}$ commutes with each entry of ${\bf Y}$.
      The results are obtained in a more general setting, namely, when ${\bf X}$ and ${\bf Y}$ belong to  noncommutative varieties $\cV_1$ and $\cV_2$ determined by row contractions subject to constraints such as
$$q(X_1,\ldots, X_{n_1})=0 \quad \text{and} \quad r(Y_1,\ldots, Y_{n_2})=0, \qquad  q\in \cP, r\in \cR,
$$
 respectively, where $\cP$ and $\cR$ are sets of noncommutative  polynomials.
We obtain dilation results which  simultaneously  generalize   Sz.-Nagy dilation theorem  for contractions, And\^ o's dilation theorem for commuting contractions, Sz.-Nagy--Foia\c s commutant lifting theorem, and  Schur's representation for the unit ball of $H^\infty$, in the framework of noncommutative varieties and Poisson kernels on Fock spaces. This leads to one of the main results of the paper, an And\^ o type inequality on noncommutative varieties, which,  in the particular case when  $n_1=n_2=1$ and $T_1$ and $T_2$ are commuting contractive matrices  with spectrum in the open unit disk $\DD:=\{z\in \CC:\ |z|<1\}$, takes the form
$$
\|p(T_1, T_2)\|\leq \min\left\{ \|p(B_1\otimes I_{\CC^{d_1}},\varphi_1(B_1))\|, \|p(\varphi_2(B_2), B_2\otimes I_{\CC^{d_2}})\|\right\},
$$
where $(B_1\otimes I_{\CC^{d_1}},\varphi_1(B_1))$ and $(\varphi_2(B_2), B_2\otimes I_{\CC^{d_2}})$ are analytic dilations of $(T_1, T_2)$ while $B_1$ and $B_2$ are the universal models  associated with $T_1$ and $T_2$, respectively. In this  setting, the inequality  is sharper than And\^ o's inequality and  Agler-McCarthy's inequality. We obtain  more general inequalities  for arbitrary commuting contractive matrices  and improve And\^ o's inequality for   commuting contractions  when at least one of them is of class $\cC_0$.

We prove that there   is a universal model $(S\otimes I_{ \ell^2},\varphi(S))$, where $S$ is the unilateral shift and $\varphi(S)$ is an  isometric  analytic Toeplitz operator on $H^2(\DD)\otimes \ell^2$,   such that
$$
\|[p_{rs}({T}_1,{T}_2)]_{k}\|\leq  \|[p_{rs}(S\otimes I_{ \ell^2},\varphi(S))]_{k }\|,
$$
 for any commuting contractions ${T}_1$ and $ {T}_2$ on Hilbert spaces,  any   $k\times k$ matrix $[p_{rs}]_{k}$ of polynomials in $\CC[z,w]$, and  any $k\in \NN$. Analogues of this result for the bi-ball
 ${\bf P}_{n_1,n_2}$ and  for a class of noncommutative varieties are also considered.
 \end{abstract}

      \maketitle

\section*{Introduction}

 Two of the most important results in operator theory are  von Neumann's inequality \cite{vN} and And\^ o's generalization \cite{An} to two commuting contractions on a Hilbert space (see also \cite{SzFBK-book}, \cite{CW1},   \cite{CW2}, and \cite{DS}). And\^ o's inequality is essentially equivalent (see \cite{NV} and \cite{Pa-book}) to the Sz.-Nagy--Foia\c s commutant lifting theorem \cite{SzF} and  states that if $T_1$ and $T_2$ are commuting contractions on a Hilbert space, then for any polynomial $p$ in two variables,
$$
\|p(T_1, T_2)\|\leq \|p\|_{\DD^2},
$$
where $\DD^2$ is the bidisk in $\CC^2$.
Varopoulus \cite{Varo} found a counterexample showing that the inequality  does not extend to three mutually commuting contractions. For a nice survey and further generalizations of these inequalities  we refer to Pisier's book \cite{Pi-book}.

In a remarkable paper \cite{AM}, Agler and McCarthy improved And\^ o's inequality in the case of contractive matrices with no eigenvalues of modulus 1. They showed that
$$
 \|p(T_1, T_2)\|\leq \|p\|_V,
$$
where $V$ is some  distinguished  variety in the bidisc $\DD^2$ depending on $T_1$ and $T_2$. This result was extended by Das and Sarkar \cite{DS} to pure commuting contractions.

To write  the present paper,  we were  inspired by the work of Agler-McCarthy \cite{AM} and Das-Sarkar \cite{DS} on distinguished varieties and by Ando's inequality for two commuting contractions.
One of the main results of our paper improves the above-mentioned inequality for commuting contractive matrices $T_1$ and $T_2$ with spectrum in the open unit disk $\DD$. The idea is the following. Each matrix $T_i$ lives in an operator-valued variety $\cV_i$, which is uniquely determined  by its minimal polynomial, and admits  a universal model $B_i$. Using noncommutative Poisson transforms \cite{Po-poisson}, we show that  the pair $(T_1, T_2)$  admits {\it analytic dilations} $(B_1\otimes I_{\CC^{d_1}},\varphi_1(B_1))$ and $(\varphi_2(B_2), B_2\otimes I_{\CC^{d_2}})$ which encode the algebraic and geometric structure of $(T_1, T_2)$. This leads to  the inequality
$$
\|[p_{rs}(T_1, T_2)]_k\|\leq \min\left\{ \|[p_{rs}(B_1\otimes I_{\CC^{d_1}},\varphi_1(B_1))]_k\|, \|[p_{rs}(\varphi_2(B_2), B_2\otimes I_{\CC^{d_2}})]_k\|\right\}
$$
 for any   $k\times k$ matrix $[p_{rs}]_{k}$,  $k\in \NN$, of polynomials in $\CC[z,w]$,
which is sharper than And\^ o's inequality,  Agler-McCarthy's inequality, and Das-Sarkar's extension (see the remarks and examples following  Theorem \ref{AM3}).
In fact, we obtain a more general result  for arbitrary commuting contractive matrices (see Theorem \ref{general}). We also  improve And\^ o's inequality for   commuting contractions  when at least one of them is of class $\cC_0$ (see Theorem \ref{A-M2} and Remark \ref{rem}).

On the other hand, we  prove that there   is a universal model $(S\otimes I_{ \ell^2},\varphi(S))$, where $S$ is the unilateral shift on the Hardy space $H^2(\DD)$  and $\varphi(S)$ is an  isometric  analytic Toeplitz operator on $H^2(\DD)\otimes \ell^2$,   such that
$$
\|[p_{rs}({T}_1,{T}_2)]_{k}\|\leq  \|[p_{rs}(S\otimes I_{ \ell^2},\varphi(S))]_{k }\|,
$$
 for any commuting contractions ${T}_1$ and $ {T}_2$ on Hilbert spaces,  any   $k\times k$ matrix $[p_{rs}]_{k}$ of polynomials in $\CC[z,w]$ and  any $k\in \NN$.
The closed non-self-adjoint algebra $\cA_u({\DD^2})$ generated by
$ S \otimes I_{\ell^2} , \varphi({S})$ and the identity  can be seen as the universal operator algebra for two commuting contractions. All these results are presented in Section 3.

The main goal of our paper  is to obtain  And\^ o dilations and inequalities  on noncommutative varieties  in the more general setting of bi-balls. In this setting, it would be interesting to find  good analogues for {\it distinguished varieties} in the sense of  \cite{AM}. For now, this remains an open problem.  To present our results we need some notation and definitions.

Throughout this paper, $B(\cH)$ stands for the algebra of all bounded linear operators on a Hilbert space $\cH$. We denote by  $B(\cH)^{n_1}\times_c  B(\cH)^{n_2}$, where $n_1, n_2 \in\NN:=\{1,2,\ldots\}$,
   the set of all pairs  $ ({\bf X},{ \bf Y})$ in $B(\cH)^{n_1}\times  B(\cH)^{n_2}$
     with the property that the entries of ${\bf X}:=(X_{1},\ldots, X_{n_1})$  are commuting with the entries of
      ${\bf Y}:=(Y_{1},\ldots, Y_{n_2})$.
   Let ${\bf n}:=(n_1, n_2)$ and define  the open  {\it bi-ball}
  $${\bf P_n}(\cH):=[B(\cH)^{n_1}]_1\times_c  [B(\cH)^{n_2}]_1,
  $$
  where
    $$[B(\cH)^{n}]_1:=\{(X_1,\ldots, X_n)\in B(\cH)^{n}:\ \|  X_1X_1^*+\cdots +X_nX_n^*\|<1\}, \quad n\in \NN.
    $$
    The closure of ${\bf P_n}(\cH)$ in the operator norm topology is denoted by ${\bf P}_{\bf n}^-(\cH)$.
For simplicity, throughout this paper, $[X_1,\ldots, X_n]$ denotes either the $n$-tuple $(X_1,\ldots, X_n)\in B(\cH)^n$ or the operator row matrix
$[X_1\cdots X_n]$ acting from $\cH^{(n)}$, the direct sum of $n$ copies of the Hilbert space $\cH$, to $\cH$.

Let $H_n$ be an $n$-dimensional complex  Hilbert space with
orthonormal
      basis
      $e_1,\ldots,e_n$, where $n\in \NN$.
       We consider the full Fock space  of $H_n$ defined by
      $$F^2(H_n):=\CC1\oplus \bigoplus_{k\geq 1} H_n^{\otimes k},$$
      where  $H_n^{\otimes k}$ is the (Hilbert)
      tensor product of $k$ copies of $H_n$.
      Define the left  (resp.~right) {\it creation
      operators}  $S_i$ (resp.~$R_i$), $i=1,\ldots,n$, acting on $F^2(H_n)$  by
      setting
      $$
       S_i\varphi:=e_i\otimes\varphi, \quad  \varphi\in F^2(H_n),
      $$
       (resp.~$
       R_i\varphi:=\varphi\otimes e_i
      $). The universal models asociated with the unit ball $[B(\cH)^n]_1$  are  ${\bf S}:=[S_1,\ldots, S_n]$ and ${\bf R}:=[R_1,\ldots, R_n]$.
The noncommutative disc algebra $\cA_n$ (resp.~$\cR_n$) is the norm
closed non-self-adjoint algebra generated by the left (resp.~right) creation
operators and the identity. The   noncommutative analytic Toeplitz
algebra $F_n^\infty$ (resp.~$\cR_n^\infty$)
 is the  weakly
closed version of $\cA_n$ (resp.~$\cR_n$). These algebras were
introduced in \cite{Po-von}, \cite{Po-funct} in connection with a noncommutative von
Neumann  type inequality.

Let $J$ be a WOT-closed two-sided ideal of $F_n^\infty$
such that $J\neq F_n^\infty$, and define the noncommutative variety $\cV_J(\cH)$ to be  the set of all {\it pure row contractions} ${\bf T}:=[T_1,\ldots, T_n]$, \ $T_i\in B(\cH)$, such that
$$
f(T_1,\ldots, T_n)=0\quad \text{ for any }\quad f\in J,
$$
where $f(T_1,\ldots, T_n)$ is defined according to the $F_n^\infty$-functional calculus (see \cite{Po-funct}).
We proved in \cite{Po-varieties}  that there are  some universal models ${\bf B}:=[B_1,\ldots, B_n]$ and ${\bf W}:=[W_1,\ldots, W_n]$  associated with the noncommutative variety $\cV_J$, which are compressions of the creation operators  to a certain  joint co-invariant subspace uniquely determined by $J$.  The   noncommutative Hardy
algebra  $\cR_n^\infty(\cV_J)$
 is the  weakly closed non-self-adjoint algebra generated by $W_1,\ldots, W_n$ and the identity.

  In Section 1, borrowing ideas from \cite{AM} and \cite{DS}, we obtain  dilation results (Theorem \ref{dil} and Theorem \ref{dil-com}), for the  elements  of  noncommutative varieties $\cV_J$, which are simultaneous multi-variable generalizations of Sz.-Nagy dilation theorem  for contractions \cite{SzFBK-book},  And\^ o's dilation theorem for commuting contractions \cite{An}, Sz.-Nagy--Foia\c s commutant lifting theorem \cite{SzFBK-book}, and of Schur's representation for the unit ball of $H^\infty$ \cite{Sc}, in the framework of noncommutative Poisson kernels associated with $\cV_J$. As consequences, we provide a Schur type representation  for the noncommutative Hardy algebra $\cR_n^\infty(\cV_J)$ (see Theorem \ref{transfer}) and a new proof for the commutant lifting theorem for pure row contractions in noncommutative varieties (see \cite{Po-isometric}, \cite{Po-varieties}).

In Section 2, using the results from Section 1, we   prove that any pair $({\bf T}_1, {\bf T}_2)$ in
the noncommutative variety
$${{\bf P}_{J_1, J_2,{\bf n}}^-}(\cH):=\left\{({\bf T}_1, {\bf T}_2)\in {\bf P}_{\bf n}^-(\cH): {\bf T}_1\in \cV_{J_1}(\cH), {\bf T}_2\in \cV_{J_2}(\cH)\right\},
$$
has analytic dilations
 $$({\bf B}_1\otimes  I_{\ell^2},  \varphi_1({\bf  W}_1))\quad \text{and} \quad
 (\varphi_2({\bf W}_2), {\bf B}_2\otimes I_{\ell^2})
 $$
 where
$\varphi_1({\bf  W}_1)$ and $\varphi_2({\bf  W}_2)$ are some  contractive  multi-analytic operators with  respect to the universal models ${\bf B}_1$ and ${\bf B}_2$ of the varieties $\cV_{J_1}$ and $\cV_{J_2}$, respectively.
As a consequence, we show that the inequality
$$
\|[p_{rs}({\bf T}_1,{\bf T}_2)]_{k}\|\leq \min \left\{ \|[p_{rs}({\bf B}_1\otimes I_{\ell^2},  \varphi_1({\bf  W}_1))]_{k}\|,  \|[p_{rs}({\varphi_2({\bf  W}_2), \bf B}_2\otimes I_{\ell^2})]_{k}\|\right\} \qquad  p_{rs}\in \CC\left<{\bf X}, {\bf Y}\right>
$$
holds
 for any    $k\in \NN$, where $\CC\left<{\bf X}, {\bf Y}\right>$ is  the complex algebra of all  polynomials in  noncommutative indeterminates  $X_{1},\ldots, X_{n_1}$ and $Y_{1},\ldots, Y_{n_2}$.

On the other hand, we   prove that the bi-ball
 ${{\bf P}_{{\bf n}}^-}(\cH)$
has a universal model $$(S_1\otimes I_{\ell^2},\ldots S_{n_1}\otimes I_{\ell^2},  \psi_1({\bf  R}),\ldots, \psi_{n_2}({\bf  R})),
 $$
 where
$\psi({\bf  R})=(\psi_1({\bf R}),\ldots, \psi_{n_2}({\bf R})$ is some contractive multi-analytic operator with  respect to the universal model ${\bf S}=[S_1,\ldots, S_{n_1}]$, and ${\bf R}=[R_1,\ldots, R_{n_1}]$.
More precisely, we show that
$$
\|[p_{rs}({\bf T}_1,{\bf T}_2)]_{k}\|\leq  \|[p_{rs}({\bf S}\otimes I_{\ell^2},  \psi({\bf  R}))]_{k}\|, \qquad  p_{rs}\in \CC\left<{\bf X}, {\bf Y}\right>,
$$
 for any  $({\bf T}_1, {\bf T}_2)\in {{\bf P}_{{\bf n}}^-}(\cH)$ and  any $k\in \NN$.
The closed non-self-adjoint algebra  $\cA({\bf P_{n}})$ generated by
$ S_1\otimes I_{\ell^2},\ldots, S_{n_1}\otimes I_{\ell^2}, \psi_1({\bf R}),\ldots, \psi_{n_2}({\bf R})$ and the identity can be seen as the universal operator algebra of  the bi-ball ${{\bf P}_{{\bf n}}^-}$. A similar result is provided for a class of noncommutative varieties in ${\bf P}_{\bf n}^-$.

We remark that  when $J=\{0\}$, the variety $\cV_{\{0\}}$ coincides with the set of all pure row contractions and the universal models are ${\bf S}=[S_1,\ldots, S_n]$ and ${\bf R}:=[R_1,\ldots, R_n]$.
 On the other hand,  in the particular case when   $J_c$ is  the ideal generated by the polynomials
$S_iS_j-S_jS_i$, \ $i,j=1,\ldots, n$,   the  universal model  of $\cV_{J_c}$ consists of   the creation operators on the symmetric Fock space $F_s^2$.
Arveson showed in  \cite{Arv2} that $F_s^2$  can be identified with  a subspace space  ${\bf H}^2$
  of analytic functions in $\BB_n$, namely,
the
reproducing kernel Hilbert space
with reproducing kernel $K_n: \BB_n\times \BB_n\to \CC$ defined by
 $$
 K_n(z,w):= {\frac {1}
{1-\langle z, w\rangle_{\CC^n}}}, \qquad z,w\in \BB_n.
$$
 The algebra $\cR_n^\infty(\cV_{J_c})$
 can be identified with  the algebra of all  multipliers  of ${\bf H}^2$.  Under this identification the creation operators  become the multiplication operators $M_{z_1},\ldots, M_{z_n}$ by the coordinate functions $z_1,\ldots, z_n$ of $\CC^n$.  All the results of this paper, concerning And\^ o dilations and inequalities,  apply to these two important particular cases.

Finally, we would like to thank the referee for helpful comments and suggestions on the paper.

\bigskip

\section{And\^ o type dilations on  noncommutative varieties}

In this section, we obtain dilation results which  simultaneously  generalize   Sz.-Nagy dilation theorem  for contractions, And\^ o's dilation theorem for commuting contractions, Sz.-Nagy--Foia\c s commutant lifting theorem, and  Schur's representation for the unit ball of $H^\infty$, in the framework of noncommutative varieties and Poisson kernels on Fock spaces.

 Let $\FF_n^+$ be the unital free semigroup on $n$ generators
$g_1,\ldots, g_n$ and the identity $g_0$.  The length of $\alpha\in
\FF_n^+$ is defined by $|\alpha|:=0$ if $\alpha=g_0$ and
$|\alpha|:=k$ if
 $\alpha=g_{i_1}\cdots g_{i_k}$, where $i_1,\ldots, i_k\in \{1,\ldots, n\}$.
If $(X_1,\ldots, X_n)\in B(\cH)^n$, where $B(\cH)$ is the algebra of
all bounded linear operators on the Hilbert space $\cH$,    we set
$X_\alpha:= X_{i_1}\cdots X_{i_k}$  and $X_{g_0}:=I_\cH$.  Let $H_n$ be a finite dimensional Hilbert space with orthonormal basis $e_1,\ldots, e_n$. We denote
$e_\alpha:= e_{i_1}\otimes\cdots \otimes  e_{i_k}$ and $e_{g_0}:=1$.
Note that $\{e_\alpha\}_{\alpha\in \FF_n^+}$ is an orthonormal basis
for the full Fock space $F^2(H_n)$.

  We recall   (\cite{Po-von},  \cite{Po-funct},
      \cite{Po-analytic})
       a few facts
       concerning multi-analytic   operators on Fock
      spaces.
         We say that
       a bounded linear
        operator
      $M$ acting from $F^2(H_n)\otimes \cK$ to $ F^2(H_n)\otimes \cK'$ is
       multi-analytic with respect to the universal model ${\bf S}:=[S_1,\ldots, S_n]$
      if
      \begin{equation*}
      M(S_i\otimes I_\cK)= (S_i\otimes I_{\cK'}) M\quad
      \text{\rm for any }\ i=1,\dots, n.
      \end{equation*}
       We can associate with $M$ a unique formal Fourier expansion
      $  \sum_{\alpha \in \FF_n^+}
      R_\alpha \otimes \theta_{(\alpha)}$,
where $\theta_{(\alpha)}\in B(\cK, \cK')$.
       We  know  that $\sum_{\alpha\in \FF_n^+}  \theta_{(\alpha)}^*\theta_{(\alpha)}\leq \|M\| I_\cK$ and
        $$M =\text{\rm SOT-}\lim_{r\to 1}\sum_{k=0}^\infty
      \sum_{|\alpha|=k}
         r^{|\alpha|} R_\alpha\otimes \theta_{(\alpha)},
         $$
         where, for each $r\in [0,1)$, the series converges in the uniform norm.
      Moreover, the set of  all multi-analytic operators in
      $B(F^2(H_n)\otimes \cK,
      F^2(H_n)\otimes \cK')$  coincides  with
      $\cR_n^\infty\bar \otimes B(\cK,\cK')$,
      the WOT-closed operator space generated by the spatial tensor
      product.
A multi-analytic operator is called {\it inner} if it is an isometry. We
remark that  similar results are valid  for  multi-analytic
operators with respect to the right creation operators $R_1,\ldots,
R_n$.

We need to   recall from \cite{Po-poisson} a few facts
about noncommutative Poisson transforms associated with row
contractions ${\bf T}:=[T_1,\ldots, T_n]$, \ $T_i\in B(\cH)$,  i.e. $T_1T_1^*+\cdots +T_nT_n^*\leq I$.
         For  each $0<r\leq 1$, define the defect operator
$\Delta_{{\bf T},r}:=(I_\cK-r^2T_1T_1^*-\cdots -r^2 T_nT_n^*)^{1/2}$ and the defect space   $\cD_{r{\bf T}}=\overline{\Delta_{{\bf T},r} \cH}$.
The {\it noncommutative Poisson  kernel} associated with ${\bf T}$ is the family
of operators
$$
K_{{\bf T},r} :\cH\to    F^2(H_n)\otimes \cD_{r{\bf T}}, \quad
0<r\leq 1,
$$
defined by
\begin{equation*}
K_{{\bf T},r}h:= \sum_{k=0}^\infty \sum_{|\alpha|=k}  e_\alpha\otimes r^{|\alpha|}
\Delta_{{\bf T},r} T_\alpha^*h,\quad h\in \cH.
\end{equation*}
When $r=1$, we denote $\Delta_{\bf T}:=\Delta_{{\bf T},1}$ and $K_{\bf T}:=K_{{\bf T},1}$.
The operators $K_{{\bf T},r}$ are isometries if $0<r<1$, and, if $r=1$,
$$
K_{\bf T}^*K_{\bf T}=I_\cK- \text{\rm SOT-}\lim_{k\to\infty} \sum_{|\alpha|=k}
T_\alpha T_\alpha^*.
$$
Thus $K_{\bf T}$ is an isometry if and only if ${\bf T}$ is a {\it pure row
 contraction},
i.e., $ \text{\rm SOT-}\lim\limits_{k\to\infty} \sum_{|\alpha|=k}
T_\alpha T_\alpha^*=0.$ A key property of the Poisson kernel
is that
\begin{equation*}
K_{T,r}(r^{|\alpha|} T_\alpha^*)=(S_\alpha^*\otimes I)K_{T,r}\qquad \text{ for any } 0<r\leq 1,\ \alpha\in \FF_n^+.
\end{equation*}
We refer to \cite{Po-poisson} and \cite{Po-Berezin-poly} for more on noncommutative Poisson transforms on
$C^*$-algebras generated by isometries.

Let
  ${\bf T}_1=[T_{1,1},\ldots, T_{1,n_1}]\in B(\cH)^{n_1}$ and   ${\bf T}'_1=[T_{1,1}',\ldots, T_{1,n_1}']\in B(\cH')^{n_1}$ be row contractions.
  Let ${\bf T}_2:=[T_{2,1},\ldots, T_{2,n_2}]$, with $T_{2,j}:\cH'\to \cH$,  be a  row contraction which intertwines ${\bf T}_1$ with ${\bf T}_1'$, i.e.
   $$
   T_{2,j}T_{1,i}'=T_{1,i}T_{2,j}
   $$
   for any $i\in \{1,\ldots, n_1\}$ and $j\in \{1,\ldots, n_2\}$. We denote by $\cI({\bf T}_1,{\bf T}_1')$ the set of all contractive intertwining operators ${\bf T}_2$  of ${\bf T}_1$ and ${\bf T}_1'$.
   A straightforward calculation reveals that
$$
\|\Delta_{{\bf T}_1}h\|^2 +\sum_{j=1}^{n_1}\|\Delta_{{\bf T}_2}T_{1,j}^*h\|^2
=
\sum_{j=1}^{n_2}\|\Delta_{{\bf T}_1'}T_{2,j}^*h\|^2 + \|\Delta_{{\bf T}_2}h\|^2
$$
for any $h\in \cH$.
 If
 the defect spaces $\cD_{{\bf T}_1}$, $\cD_{{\bf T}_1'}$,  and $\cD_{{\bf T}_2}$ are finite dimensional with  dimensions $d_1:=\dim \cD_{{\bf T}_1}$,$d_1':=\dim \cD_{{\bf T}_1'}$, $d_2:=\dim \cD_{{\bf T}_2}$, and
 $$
  d_1+n_1d_2=n_2 d_1'+d_2,
  $$
  then    there are unitary extensions $U:\cD_{{\bf T}_1}\oplus \bigoplus_{j=1}^{n_1}\cD_{{\bf T}_2} \to
\bigoplus_{j=1}^{n_2}\cD_{{\bf T}_1'}\oplus \cD_{{\bf T}_2}$ of the isometry
\begin{equation}
\label{iso1}
U\left(\Delta_{{\bf T}_1}h, \Delta_{{\bf T}_2}T_{1,1}^*h, \ldots, \Delta_{{\bf T}_2}T_{1,n_1}^*h\right):=
\left(\Delta_{{\bf T}_1'}T_{2,1}^*h, \ldots, \Delta_{{\bf T}_1'}T_{2,n_2}^*h,
\Delta_{{\bf T}_2}h\right),\qquad h\in \cH.
\end{equation}
We denote by $\cU_{\bf T}$ the set of all  unitary extensions of the isometry given by relation \eqref{iso1}.

In case the above-mentioned dimensional conditions are not satisfied, then   let $\cK$ be  an infinite dimensional Hilbert space and note that
the operator defined by
\begin{equation}
\label{iso2}
U\left(\Delta_{{\bf T}_1}h, \Delta_{{\bf T}_2}T_{1,1}^*h, 0,\ldots, \Delta_{{\bf T}_2}T_{1,n_1}^*h, 0\right):=
\left(\Delta_{{\bf T}_1'}T_{2,1}^*h, \ldots, \Delta_{{\bf T}_1'}T_{2,n_2}^*h,
\Delta_{{\bf T}_2}h, 0\right)
\end{equation}
is an isometry which can be extended to a unitary operator
$$U:\cD_{{\bf T}_1}\oplus \bigoplus_{j=1}^{n_1}(\cD_{{\bf T}_2}\oplus \cK)\to  \bigoplus_{j=1}^{n_2}\cD_{{\bf T}_1'}\oplus (\cD_{{\bf T}_2}\oplus \cK).
$$
 In this case, we denote by $\cU_{\bf T}^\cK$ the set of all unitary extensions of the isometry defined by \eqref{iso2}.
Let
$ U=\left[ \begin{matrix} A&B\\C&D \end{matrix}\right]
$
be the operator matrix representation of $U\in \cU_{\bf T}^\cK$, where
\begin{equation}\label{ABCD}\begin{split}
 A&:\cD_{{\bf T}_1}\to \bigoplus_{j=1}^{n_2}\cD_{{\bf T}_1'},\\
 B&:\bigoplus_{j=1}^{n_1}(\cD_{{\bf T}_2}\oplus \cK)\to \bigoplus_{j=1}^{n_2}\cD_{{\bf T}_1'},\\
 C&: \cD_{{\bf T}_1}\to \cD_{{\bf T}_2}\oplus \cK, \text{\rm and }\\
 D&: \bigoplus_{j=1}^{n_1}(\cD_{{\bf T}_2}\oplus \cK)\to \cD_{{\bf T}_2}\oplus \cK.
 \end{split}
 \end{equation}
 Given an operator $Z:\cN\to \cM$ and $n\in \NN$, we introduce the ampliation
$$
\text{\rm diag}_{n}(Z):=\left(\begin{matrix}
Z&\cdots &0\\
\vdots&\ddots& \vdots\\
0&\cdots&Z
\end{matrix} \right):\bigoplus_{s=1}^n \cN\to \bigoplus_{s=1}^n \cM.
$$
 We also use  the operator column notation
$\left[\begin{matrix}X_{(\alpha)}\\
 \vdots\\
   |\alpha|=k\end{matrix}\right]$, where
  the entries $X_{(\alpha)}$ are  arranged in the lexicographic order of the free semigroup $\FF_{n_1}^+$, that is $g_0, g_1,\ldots, g_{n_1},
g_1g_1,\ldots, g_1g_{n_1}, \ldots,$ $ g_{n_1}g_1,\ldots, g_{n_1}g_{n_1}, \ldots$ and so on.

In what follows, we need the following  technical result.
\begin{lemma} \label{le1}
\begin{equation*}
\begin{split}
&\text{\rm diag}_{n_1}\left(D\text{\rm diag}_{n_1}\left(\cdots
D\text{\rm diag}_{n_1}\left(\widehat\Delta_{{\bf T}_2}\right)
 {\bf T}_1^*\cdots\right){\bf T}_1^*\right){\bf T}_1^*\\
 &\qquad\qquad\qquad =
 \text{\rm diag}_{n_1}\left(D\text{\rm diag}_{n_1}\left(\cdots
D\text{\rm diag}_{n_1}\left(\widehat\Delta_{{\bf T}_2}\right)
 \cdots\right)\right)\left[\begin{matrix}T_{1,\alpha}^*\\
 \vdots\\
 \alpha\in \FF_{n_1}^+, |\alpha|=p\end{matrix}\right],
\end{split}
\end{equation*}
where $\text{\rm diag}_{n_1}$ appears  $p$ times on each side of the equality, and $\widehat\Delta_{{\bf T}_2}:=\left[\begin{matrix} \Delta_{{\bf T}_2}\\0\end{matrix}\right]:\cH\to \cD_{{\bf T}_2}\oplus \cK$.
\end{lemma}
\begin{proof}
Let $D=[D_1,\ldots, D_{n_1}]$ with $D_i\in B(\cD_{{\bf T}_2}\oplus \cK)$, and ${\bf T}_1:=[T_{1,1},\ldots, T_{1,n_1}]$. Note that
$$
D\text{\rm diag}_{n_1}\left(\widehat\Delta_{{\bf T}_2}\right)
 {\bf T}_1^*=\sum_{i=1}^{n_1} D_i\widehat\Delta_{{\bf T}_2} T_{1,i}^*
 $$
and
$$
\text{\rm diag}_{n_1}\left(D\text{\rm diag}_{n_1}\left(\widehat\Delta_{{\bf T}_2}\right){\bf T}_1^*\right) {\bf T}_1^*
=
\left[\begin{matrix} \sum_{i=1}^{n_1} D_i\widehat\Delta_{{\bf T}_2} (T_{1,1}T_{1,i})^*\\
\vdots\\
\sum_{i=1}^{n_1} D_i\widehat\Delta_{{\bf T}_2} (T_{1,n_1}T_{1,i})^*
\end{matrix}\right].
$$
An inductive argument shows that
$$
\text{\rm diag}_{n_1}\left(D\text{\rm diag}_{n_1}\left(\cdots
D\text{\rm diag}_{n_1}\left(\widehat\Delta_{{\bf T}_2}\right)
 {\bf T}_1^*\cdots\right){\bf T}_1^*\right){\bf T}_1^*
 =
 \left[\begin{matrix} \sum\limits_{\alpha\in \FF_{n_1}^+, |\alpha|=p-1} D_\alpha\widehat\Delta_{{\bf T}_2} (T_{1,g_1}T_{1,\alpha})^*\\
\vdots\\
\sum\limits_{\alpha\in \FF_{n_1}^+, |\alpha|=p-1} D_\alpha\widehat\Delta_{{\bf T}_2} (T_{1,g_{n_1}}T_{1,\alpha})^*
\end{matrix}\right],
$$
where $\text{\rm diag}_{n_1}$ appears  $p$ times.
On the other hand, one can easily prove by induction that
\begin{equation*}
\begin{split}
&\text{\rm diag}_{n_1}\left(D\text{\rm diag}_{n_1}\left(\cdots
D\text{\rm diag}_{n_1}\left(\widehat\Delta_{{\bf T}_2}\right)
 \cdots\right)\right)
 \left[\begin{matrix}T_{1,\alpha}^*\\
 \vdots\\
  |\alpha|=p\end{matrix}\right]\\
  &\qquad \qquad \qquad
 =\text{\rm diag}_{n_1}\left([D_\alpha \widehat\Delta_{{\bf T}_2}: \ \alpha\in \FF_{n_1}^+, |\alpha|=p-1]\right)\left[\begin{matrix}
  \left[\begin{matrix}T_{1,g_1\alpha}^*\\
 \vdots\\
  |\alpha|=p-1\end{matrix}\right]\\
  \vdots\\
  \left[\begin{matrix}T_{1,g_{n_1}\alpha}^*\\
 \vdots\\
  |\alpha|=p-1\end{matrix}\right]
  \end{matrix}\right].
 \end{split}
 \end{equation*}
 The proof is complete.
  \end{proof}

 \begin{lemma} \label{series} Let ${\bf T}_2\in \cI({\bf T}_1,{\bf T}_1')$
     and let
$ U=\left[ \begin{matrix} A&B\\C&D \end{matrix}\right]
$
be the matrix representation of a unitary extension  $U\in \cU_{\bf T}^\cK$ (see relation \eqref{ABCD}).
If ${\bf T}_1$ is a pure row contraction,  then
 \begin{equation*}
  \begin{split}
  \text{\rm diag}_{n_2}(\Delta_{{\bf T}_1'}) {\bf T}_2^*h
=A\Delta_{{\bf T}_1}h
+B\sum_{p=0}^\infty \text{\rm diag}_{n_1}\left(D\text{\rm diag}_{n_1}\left(\cdots
D\text{\rm diag}_{n_1}\left(C\Delta_{{\bf T}_1}\right) {\bf T}_1^*\cdots\right){\bf T}_1^*\right){\bf T}_1^*h,
 \end{split}
  \end{equation*}
for any $h\in \cH$, where $\text{\rm diag}_{n_1}$ appears $p+1$ times in the general term of  the series.
 \end{lemma}
 \begin{proof}
 Due to relation \eqref{iso2}, we have
\begin{equation}\label{A}
A\Delta_{{\bf T}_1}h +B\left[\begin{matrix}
\Delta_{{\bf T}_2}T_{1,1}^*h\\0\\\vdots\\ \Delta_{{\bf T}_2}T_{1,n_1}^*h\\ 0
\end{matrix}\right]
=
\left[\begin{matrix}
\Delta_{{\bf T}_1'}T_{2,1}^*h\\\vdots\\ \Delta_{{\bf T}_1'}T_{2,n_2}^*h
\end{matrix}\right]
\end{equation}
and
\begin{equation}\label{C}
C\Delta_{{\bf T}_1}h +D\left[\begin{matrix}
\Delta_{{\bf T}_2}T_{1,1}^*h\\0\\\vdots\\ \Delta_{{\bf T}_2}T_{1,n_1}^*h\\ 0
\end{matrix}\right]
=
\left[\begin{matrix}
\Delta_{{\bf T}_2} h\\0
\end{matrix}\right]
\end{equation}
for any $h\in \cH$.
Since $\widehat\Delta_{{\bf T}_2}:=\left[\begin{matrix} \Delta_{{\bf T}_2}\\0\end{matrix}\right]:\cH\to \cD_{{\bf T}_2}\oplus \cK$, we can rewrite relations \eqref{A}  and \eqref{C} as
 \begin{equation}
 \label{AA}
 A\Delta_{{\bf T}_1}+B\text{\rm diag}_{n_1}(\widehat\Delta_{{\bf T}_2}) {\bf T}_1^*
 =
 \text{\rm diag}_{n_2}(\Delta_{{\bf T}_1'}) {\bf T}_2^*
 \end{equation}
and
\begin{equation}
 \label{CC}
 C\Delta_{{\bf T}_1}+D\text{\rm diag}_{n_1}(\widehat\Delta_{{\bf T}_2}) {\bf T}_1^*
 =
 \widehat\Delta_{{\bf T}_2},
 \end{equation}
respectively. Note that  using relation \eqref{CC}  we deduce that
\begin{equation}
\label{Di}
\text{\rm diag}_{n_1}\left(\widehat\Delta_{{\bf T}_2}\right){\bf T}_1^*=\text{\rm diag}_{n_1}\left(C\Delta_{{\bf T}_1}\right){\bf T}_1^*
+\text{\rm diag}_{n_1}\left(D\text{\rm diag}_{n_1}(\widehat\Delta_{{\bf T}_2}) {\bf T}_1^*\right){\bf T}_1^*,
\end{equation}
which combined with relation \eqref{AA} yields

\begin{equation*}
\begin{split}
\text{\rm diag}_{n_2}(\Delta_{{\bf T}_1'}) {\bf T}_2^*
=A\Delta_{{\bf T}_1}+B \text{\rm diag}_{n_1}\left(C\Delta_{{\bf T}_1}\right){\bf T}_1^*
+B\text{\rm diag}_{n_1}\left(D\text{\rm diag}_{n_1}(\widehat\Delta_{{\bf T}_2}) {\bf T}_1^*\right){\bf T}_1^*.
\end{split}
\end{equation*}
  Continuing to use relation \eqref{Di} in the latter relation  and the resulting ones, an induction argument leads to  the identity
  \begin{equation}\label{rel}
  \begin{split}
  \text{\rm diag}_{n_2}(\Delta_{{\bf T}_1'}) {\bf T}_2^*
=A\Delta_{{\bf T}_1}&+B \text{\rm diag}_{n_1}\left(C\Delta_{{\bf T}_1}\right){\bf T}_1^*\\
&+B\sum_{p=1}^m \text{\rm diag}_{n_1}\left(D\text{\rm diag}_{n_1}\left(\cdots
D\text{\rm diag}_{n_1}\left(C\Delta_{{\bf T}_1}\right) {\bf T}_1^*\cdots\right){\bf T}_1^*\right){\bf T}_1^* \\
&+B\text{\rm diag}_{n_1}\left(D\text{\rm diag}_{n_1}\left(\cdots
D\text{\rm diag}_{n_1}\left(\widehat\Delta_{{\bf T}_2}\right)
 {\bf T}_1^*\cdots\right){\bf T}_1^*\right){\bf T}_1^*,
  \end{split}
  \end{equation}
where $\text{\rm diag}_{n_1}$ appears $p+1$ times in the general term of  the sum above and $m+2$ times in the last term.
Since $\Delta_{{\bf T}_2}$ and $D$ are contractions and using Lemma \ref{le1}, one can easily see that
$$
\left\|
B\text{\rm diag}_{n_1}\left(D\text{\rm diag}_{n_1}\left(\cdots
D\text{\rm diag}_{n_1}\left(\widehat\Delta_{{\bf T}_2}\right)
 {\bf T}_1^*\cdots\right){\bf T}_1^*\right){\bf T}_1^*h\right\|
 \leq \|B\|\left(\sum_{\alpha\in \FF_{n_1}^+, |\alpha|=m+2}\|T_{1,\alpha}^*h\|^2\right)^{1/2}
 $$
 for any $h\in \cH$.
 Since ${\bf T}_1$ is a pure row contraction, we have
 $\lim_{m\to \infty}\sum_{\alpha\in \FF_{n_1}^+, |\alpha|=m+2}\|T_{1,\alpha}^*h\|^2=0$ for any $h\in \cH$.
Consequently, relation \eqref{rel} implies
\begin{equation*}
  \begin{split}
  \text{\rm diag}_{n_2}(\Delta_{{\bf T}_1'}) {\bf T}_2^*h
=A\Delta_{{\bf T}_1}h
+B\sum_{p=0}^\infty \text{\rm diag}_{n_1}\left(D\text{\rm diag}_{n_1}\left(\cdots
D\text{\rm diag}_{n_1}\left(C\Delta_{{\bf T}_1}\right) {\bf T}_1^*\cdots\right){\bf T}_1^*\right){\bf T}_1^*h,
 \end{split}
  \end{equation*}
for any $h\in \cH$, where $\text{\rm diag}_{n_1}$ appears $p+1$ times in the general term of  the series.
The proof is complete.
\end{proof}

\begin{lemma} \label{le2} If
$ U=\left[ \begin{matrix} A&B\\C&D \end{matrix}\right]
$
is  the matrix representation of a unitary extension  $U\in \cU_{\bf T}^\cK$, then
\begin{equation*}
\begin{split}
\text{\rm diag}_{n_1}\left(\text{\rm diag}_{n_1}\cdots\left(
\text{\rm diag}_{n_1}\left(C^*\right)
D^*\right)\cdots D^*\right) =
 \text{\rm diag}_{n_1}\left(\left[\begin{matrix}C^*D_{ \gamma}^*\\
 \vdots\\
 \gamma\in \FF_{n_1}^+, |\gamma|=q\end{matrix}\right]\right),
\end{split}
\end{equation*}
where $\text{\rm diag}_{n_1}$ appears  $q+1$ times on the left-hand side of the equality.
\end{lemma}
\begin{proof}  As in Lemma \ref{le1}, an induction argument  over $q$ can be used to complete the proof.
\end{proof}

Let $\cH$, $\cH'$, and $\cE$ be Hilbert spaces and
consider
$$ U=\left[ \begin{matrix} A&B\\C&D \end{matrix}\right]:\begin{matrix}\cH\\\oplus \\ \bigoplus_{j=1}^{n_1}\cE\end{matrix}\to  \begin{matrix} \bigoplus_{j=1}^{n_2}\cH'\\\oplus \\\cE\end{matrix}
$$
to be any unitary operator. Set
$$D=[D_1,\ldots, D_{n_1}]: \bigoplus_{j=1}^{n_1}\cE\to \cE
$$
and let ${\bf R}:=[R_{1,1},\ldots, R_{1,n_1}]$ be the tuple of right creation operators on the full Fock space $F^2(H_{n_1})$.
We associate with $U^*$  and any $r\in [0,1)$ the operator
$\varphi_{U^*}(r{\bf R})$ defined by
\begin{equation*}
\begin{split}
\varphi_{U^*}(r{\bf R}):= I_{F^2(H_{n_1})}\otimes A^*&+\left(I_{F^2(H_{n_1})}\otimes C^*\right) \left(I_{F^2(H_{n_1})\otimes \cE}-\sum_{j=1}^{n_1}r R_{1,j}\otimes D_j^*\right)^{-1}\\
& \times \left[r R_{1,1}\otimes I_{\cE},\ldots,
rR_{1,n_1}\otimes I_{\cE}\right]\left(I_{F^2(H_{n_1})}\otimes B^*\right).
\end{split}
\end{equation*}
We will see the  proof of the next lemma that the operator $\varphi_{U^*}(r{\bf R})$ is well-defined.
In what follows we use the notations: ${\bf A}:=I_{F^2(H_{n_1})}\otimes A$, ${\bf B}:=I_{F^2(H_{n_1})}\otimes B$,
 ${\bf C}:=I_{F^2(H_{n_1})}\otimes C$, ${\bf D}:=I_{F^2(H_{n_1})}\otimes D$,
    and
 ${\bf \Gamma}:=\left[ R_{1,1}\otimes I_{\cE},\ldots,
R_{1,n_1}\otimes I_{\cE}\right]$.

\begin{lemma} \label{strong-limit} The strong operator topology limit
$$\varphi_{U^*}({\bf R}):=\text{\rm SOT-}\lim_{r\to 1}\varphi_{U^*}(r{\bf R})
$$
 exists and defines a contractive  multi-analytic  operator
 $\varphi_{U^*}({\bf R})=(\varphi_1({\bf R}),\ldots, \varphi_{n_2}({\bf R}))$ with  $\varphi_j({\bf R})\in \cR_{n_1}^\infty\bar\otimes B\left( \cH',\cH\right)$, where $\cR_{n_1}^\infty $ is the noncommutative analytic Toeplitz  algebra generated by the right creation operators $R_{1,1},\ldots, R_{1,n_1}$  on the full Fock space $F^2(H_{n_1})$ and the identity.

Moreover,
$\varphi_{U^*}({\bf R})$ is an isometry if and only if
\begin{equation*}
\lim_{r\to 1}(1-r^2)\|(I-r{\bf   {\bf D}^*\Gamma})^{-1}{\bf B}^*x\|^2=0, \qquad    x\in F^2(H_{n_1})\otimes \bigoplus_{j=1}^{n_2}\cH'.
\end{equation*}
\end{lemma}

\begin{proof} First, we remark that  $D$ is a contraction,  $R_{1,1},\ldots, R_{1,n_1}$ are isometries with orthogonal ranges, and  $\left\|\sum_{j=1}^{n_1}r R_{1,j}\otimes D_j^*\right\|\leq r<1$. Consequently, the operator
$I_{F^2(H_{n_1})\otimes \cE}-\sum_{j=1}^{n_1}r R_{1,j}\otimes D_j^*$  is invertible  and
$\varphi_{U^*}(r{\bf R})=(\varphi_1(r{\bf R}),\ldots, \varphi_{n_2}(r{\bf R}))$ with    $\varphi_j(r{\bf R})\in \cR_{n_1}\bar\otimes B\left( \cH',\cH\right)$,
where $\cR_{n_1}$ is the noncommutative disc algebra generated by the right creation operators $R_{1,1},\ldots, R_{1,n_1}$ and the identity.
Note also that
\begin{equation}\label{fi-rep}
\varphi_j(r{\bf R})=\sum_{k=0}^\infty \sum_{\alpha\in \FF_{n_1}^+, |\alpha|=k} r^{|\alpha|} R_{1,\alpha}\otimes \Theta_{(\alpha)}^{(j)}
\end{equation}
for some operators $\Theta_{(\alpha)}^{(j)}\in B(\cH',\cH)$, where the convergence is in the operator norm topology.

On the other hand, since $ \left[ \begin{matrix} {\bf A}^*&{\bf C}^*\\{\bf B}^*&{\bf  D}^* \end{matrix}\right]
$
is a unitary operator, standard calculations (see e.g. \cite{Po-varieties})    show that
\begin{equation*}
\begin{split}
I-\varphi_{U^*}(r{\bf R})\varphi_{U^*}(r{\bf R})^*&= (1-r^2){\bf C}^*(I-r{\bf  \Gamma} {\bf D}^*)^{-1} (I-r{\bf   {\bf D}\Gamma}^*)^{-1}{\bf C} \\
I-\varphi_{U^*}(r{\bf R})^*\varphi_{U^*}(r{\bf R})&= (1-r^2){\bf B}(I-r{\bf  \Gamma}^* {\bf D})^{-1} (I-r{\bf   {\bf D}^*\Gamma})^{-1}{\bf B}^*.
\end{split}
\end{equation*}
The first of the relations above   shows that $\varphi_{U^*}(r{\bf R})$ is a contraction for any $r\in [0,1)$. Consequently, since  $\varphi_{U^*}(r{\bf R})=(\varphi_1(r{\bf R}),\ldots, \varphi_{n_2}(r{\bf R}))$ and  $\varphi_j(r{\bf R})\in \cR_{n_1}\bar\otimes B\left(\cH',\cH\right)
$
 has the  Fourier representation \eqref{fi-rep}, one can see
 that
  $\varphi_j({\bf R}):=\text{\rm SOT-}\lim_{r\to 1}\varphi_j(r{\bf R})$ exists,
  $\varphi_j({\bf R})\in \cR_{n_1}^\infty\bar\otimes B\left(\cH',\cH\right)$,  and $\|\varphi_{U^*}({\bf R})\|\leq 1$.

Since $\varphi_{U^*}({\bf R})=\text{\rm SOT-}\lim_{r\to 1}\varphi_{U^*}(r{\bf R})$, the second of the relations above shows that
$\varphi_{U^*}({\bf R})$ is an isometry if and only if
\begin{equation} \label{cond-iso}
\lim_{r\to 1}(1-r^2)\|(I-r{\bf   {\bf D}^*\Gamma})^{-1}{\bf B}^*x\|^2=0, \qquad x\in F^2(H_{n_1})\otimes \bigoplus_{j=1}^{n_2} \cH'.
\end{equation}
The proof is complete.
\end{proof}

 One of the main results of this section is the following  intertwining  dilation theorem for  the elements of $\cI({\bf T}_1,{\bf T}_1')$, the set of all intertwining row contractions of ${\bf T}_1$ and ${\bf T}_1'$.

\begin{theorem} \label{dil}  Let
  ${\bf T}_1:=(T_{1,1},\ldots, T_{1,n_1})$, ${\bf T}_1':=(T_{1,1}',\ldots, T_{1,n_1}')$, and ${\bf T}_2:=(T_{2,1},\ldots, T_{2,n_2})$ be row contractions such that ${\bf T}_2\in \cI({\bf T}_1,{\bf T}_1')$, and let ${\bf S}:=[S_{1,1},\ldots, S_{1,n_1}]$ and ${\bf R}:=[R_{1,1},\ldots, R_{1,n_1}]$ be the creation operators on the full Fock space $F^2(H_{n_1})$.
If
$\varphi_{U^*}({\bf R})=(\varphi_1({\bf R}),\ldots, \varphi_{n_2}({\bf R}))$ is the  multi-analytic operator associated with $U\in \cU_{\bf T}^\cK$ and
   ${\bf T}_1$ is a  pure row contraction, then the following relations hold:
$$
K_{{\bf T}_1'} T_{2,j}^* =\varphi_j({\bf R})^*K_{{\bf T}_1}, \qquad j\in\{1,\ldots, n_2\}
$$
and
$$
K_{{\bf T}_1}T_{1,i}^*=
\left(S_{1,i}^*\otimes I_{\cD_{{\bf T}_1}}\right)  K_{{\bf T}_1},\quad
K_{{\bf T}_1'}(T_{1,i}')^*=
\left(S_{1,i}^*\otimes I_{\cD_{{\bf T}_1'}}\right)  K_{{\bf T}_1'}, \qquad i\in \{1,\ldots, n_1\},
$$
 where  $K_{{\bf T}_1}$ and $K_{{\bf T}_1'}$ are the   noncommutative Poisson kernels associated with ${\bf T}_1$ and ${\bf T}_1'$, respectively.
\end{theorem}
\begin{proof} Fix $U=\left[ \begin{matrix} A&B\\C&D \end{matrix}\right]\in \cU_{\bf T}^\cK$.
We use the notations: ${\bf A}:=I_{F^2(H_{n_1})}\otimes A$, ${\bf B}:=I_{F^2(H_{n_1})}\otimes B$,
 ${\bf C}:=I_{F^2(H_{n_1})}\otimes C$,
 ${\bf Q}:=\sum_{j=1}^{n_1} R_{1,j}\otimes D_j^*$,  and
 ${\bf \Gamma}:=\left[ R_{1,1}\otimes I_{\cD_{{\bf T}_2}\oplus \cK},\ldots,
R_{1,n_1}\otimes I_{\cD_{{\bf T}_2}\oplus \cK}\right]$.
We associate with $U^*$   the  multi-analytic operator $\varphi_{U^*}({\bf R})\in \cR_{n_1}^\infty\otimes B\left(\bigoplus_{j=1}^{n_2}\cD_{{\bf T}_1'}, \cD_{{\bf T}_1}\right)$,  as in Lemma \ref{strong-limit}, in the particular case when $\cH=\cD_{{\bf T}_1}$, $\cH'=\cD_{{\bf T}_1'}$, and $\cE=\cD_{{\bf T}_2}\oplus \cK$.
Note that, for each $y\in \oplus_{j=1}^{n_2} \cD_{{\bf T}_1'}$ and $\alpha\in \FF_{n_1}^+$ with $|\alpha|=n$, we have

$$
\varphi_{U^*}({\bf R})(e_\alpha\otimes y)=e_\alpha\otimes A^*y+
\sum_{q=0}^\infty {\bf C}^*{\bf Q}^q {\bf \Gamma}{\bf B}^*(e_\alpha\otimes y)
$$
and ${\bf C}^*{\bf Q}^q {\bf \Gamma}{\bf B}^*(e_\alpha\otimes y)$ is in the closed linear span of all the vectors $e_{\alpha\sigma}\otimes z$, where $\sigma\in \FF_{n_1}^+, |\sigma|=q+1$ and $z\in \cD_{{\bf T}_1}$. Consequently, using the noncommutative Poisson kernel $K_{{\bf T}_1}$,  we  deduce that
\begin{equation*}
\begin{split}
\left<\varphi_{U^*}({\bf R})^*K_{{\bf T}_1}h, e_\alpha\otimes y\right>
&=\left< \sum_{k=0}^\infty \sum_{|\beta|=k}
  e_\beta\otimes \Delta_{{\bf T}_1}T_{1,\beta}^*h, \varphi_{U^*}({\bf R})(e_\alpha\otimes y)\right>\\
&=
\left<\Delta_{{\bf T}_1}T_{1,\alpha}^*h, A^*y\right>
+\sum_{q=0}^\infty \left< \sum_{|\sigma|=q+1}e_{\alpha\sigma}\otimes \Delta_{{\bf T}_1}T_{1,\sigma}^* T_{1,\alpha}^*h, {\bf C}^*{\bf Q}^q {\bf \Gamma}{\bf B}^*(e_\alpha\otimes y)
\right>,
\end{split}
\end{equation*}
for any $h\in \cH$,  $y\in \oplus_{j=1}^{n_2} \cD_{{\bf T}_1'}$ and $\alpha\in \FF_{n_1}^+$ with $|\alpha|=n$. If $\alpha=g_{i_1}\cdots g_{i_k}\in \FF_{n_1}^+$ , we denote by $\widetilde\alpha:=g_{i_k}\cdots g_{i_1}$ the reverse of $\alpha$.
Setting $$
D:=[D_1,\ldots, D_{n_1}]:\bigoplus_{j=1}^{n_1}(\cD_{{\bf T}_2}\oplus \cK)\to \cD_{{\bf T}_2}\oplus \cK$$ and
$$
B:=[B_1,\ldots, B_{n_1}]:\bigoplus_{j=1}^{n_1}(\cD_{{\bf T}_2}\oplus \cK)\to \bigoplus_{j=1}^{n_2}\cD_{{\bf T}_1'},
$$
we obtain
\begin{equation*}
\begin{split}
{\bf C}^*{\bf Q}^q {\bf \Gamma}{\bf B}^*(e_\alpha\otimes y)
&=
(I_{F^2(H_{n_1})}\otimes C^*)\left(\sum_{\sigma\in \FF_{n_1}^+,|\sigma|=q}
R_{1,\sigma}\otimes D^*_{\widetilde\sigma}\right)\\
&\qquad\qquad \qquad
\left[ R_{1,1}\otimes I_{\cD_{{\bf T}_2}\oplus \cK},\ldots,
R_{1,n_1}\otimes I_{\cD_{{\bf T}_2}\oplus \cK}\right]
\left[\begin{matrix} I_{F^2(H_{n_1})}\otimes B_1^*\\
\vdots\\I_{F^2(H_{n_1})}\otimes B_{n_1}^*\end{matrix}\right](e_\alpha\otimes y)\\
&=
\sum_{i=1}^{n_1} \sum_{\sigma\in \FF_{n_1}^+,|\sigma|=q}
R_{1,\sigma}R_{1,i} e_\alpha\otimes C^*D_{\widetilde\sigma}^*B_i^*y\\
&=
\sum_{i=1}^{n_1} \sum_{\sigma\in \FF_{n_1}^+,|\sigma|=q} e_{\alpha g_i\widetilde\sigma}\otimes C^*D_{\widetilde\sigma}^*B_i^*y,
\end{split}
\end{equation*}
where $\widetilde\sigma$ is the reverse of $\sigma\in \FF_{n_1}^+$.
Summing out  our findings, we deduce that
\begin{equation*}
\begin{split}
&\left<\varphi_{U^*}({\bf R})^*K_{{\bf T}_1}h, e_\alpha\otimes y\right>\\
&=
\left<\Delta_{{\bf T}_1}T_{1,\alpha}^*h, A^*y\right>+\sum_{q=0}^\infty \left< \sum_{|\sigma|=q+1}e_{\alpha\sigma}\otimes \Delta_{{\bf T}_1}T_{1,\sigma}^* T_{1,\alpha}^*h,
\sum_{i=1}^{n_1} \sum_{\gamma\in \FF_{n_1}^+,|\gamma|=q} e_{\alpha g_i\widetilde\gamma}\otimes C^*D_{\widetilde\gamma}^*B_i^*y
\right>\\
&=
\left<\Delta_{{\bf T}_1}T_{1,\alpha}^*h, A^*y\right>
+\sum_{q=0}^\infty\sum_{i=1}^{n_1}\sum_{\gamma\in \FF_{n_1}^+,|\gamma|=q}
\left< e_{\alpha g_i \widetilde\gamma}\otimes \Delta_{{\bf T}_1} T_{1, g_i\widetilde\gamma}^* T_{1,\alpha}^*h,
e_{\alpha g_i\widetilde\gamma}\otimes C^*D_{\widetilde\gamma}^*B_i^*y
\right>\\
&=
\left<\Delta_{{\bf T}_1}T_{1,\alpha}^*h, A^*y\right>+\sum_{q=0}^\infty\sum_{i=1}^{n_1}\sum_{\gamma\in \FF_{n_1}^+,|\gamma|=q}
\left<  \Delta_{{\bf T}_1} T_{1, g_i\gamma}^* T_{1,\alpha}^*h,
 C^*D_{\gamma}^*B_i^*y
\right>\\
&=
\left<\Delta_{{\bf T}_1}T_{1,\alpha}^*h, A^*y\right>
+
\sum_{q=0}^\infty\left< \left[\begin{matrix}\left[\begin{matrix} \Delta_{{\bf T}_1} T_{1, \gamma}^* T_{1,1}^*\\
\vdots\\|\gamma|=q\end{matrix}\right]T_{1,\alpha}^*h\\
\vdots \\
\left[\begin{matrix} \Delta_{{\bf T}_1} T_{1, \gamma}^* T_{1,n_1}^*\\
\vdots\\|\gamma|=q\end{matrix}\right]T_{1,\alpha}^*h
\end{matrix} \right],
\left[\begin{matrix}\left[\begin{matrix} C^* D_{\gamma}^*B_1^*y \\
\vdots\\|\gamma|=q\end{matrix}\right]\\
\vdots \\
\left[\begin{matrix} C^* D_{\gamma}^*B_{n_1}^*y\\
\vdots\\|\gamma|=q\end{matrix}\right]
\end{matrix} \right]
\right>\\
&=
\left<\Delta_{{\bf T}_1}T_{1,\alpha}^*h, A^*y\right> +
\sum_{q=0}^\infty\left<
\left[\begin{matrix} \Delta_{{\bf T}_1} T_{1, \omega}^* T_{1,\alpha}^*\\
\vdots\\|\omega|=q+1\end{matrix}\right]h,
\text{\rm diag}_{n_1}\left(\text{\rm diag}_{n_1}\cdots\left(
\text{\rm diag}_{n_1}\left(C^*\right)
D^*\right)\cdots D^*\right)B^*y
\right>\\
&=
\left<\Delta_{{\bf T}_1}T_{1,\alpha}^*h, A^*y\right> +
\left< B\sum_{p=0}^\infty \text{\rm diag}_{n_1}\left(D\text{\rm diag}_{n_1}\left(\cdots
D\text{\rm diag}_{n_1}\left(C\Delta_{{\bf T}_1}\right) {\bf T}_1^*\cdots\right){\bf T}_1^*\right){\bf T}_1^*T_{1,\alpha}^*h,y\right>.
\end{split}
\end{equation*}
The last two equalities are due to   Lemma \ref{le2}. Now, using  Lemma \ref{series}, we obtain
$$
\left<\varphi_{U^*}({\bf R})^*K_{{\bf T}_1}h, e_\alpha\otimes y\right>
=\left< \text{\rm diag}_{n_2}(\Delta_{{\bf T}_1'}){\bf T}_2^* T_{1,\alpha}^*h, y\right>
$$
for any $\alpha\in \FF_{n_1}^+$, $h\in \cH$,  and $y\in \oplus_{j=1}^{n_2}\cD_{{\bf T}_1'}$.
Hence, using the definition of the noncommutative Poisson kernel and the fact that ${\bf T}_2\in \cI({\bf T}_1,{\bf T}_1')$, we deduce that, for any $\alpha\in \FF_{n_1}^+$,  $h\in \cH$ and $y_j\in \cD_{{\bf T}_1'}$, $j=1,\ldots, n_2$,
\begin{equation*}
\begin{split}
\left<K_{{\bf T}_1'}T_{2,j}^*h, e_\alpha\otimes y_j\right>
&=\left<
\sum_{k=0}^\infty \sum_{\beta\in \FF_{n_1}^+|\beta|=k}  e_\beta\otimes
\Delta_{{\bf T}_1'} (T_{1,\beta}')^*T_{2,j}^*h,e_\alpha\otimes y_j\right>\\
&=
\left<\Delta_{{\bf T}_1'} (T_{1,\alpha}')^*T_{2,j}^*h,y_j\right>
=
\left<\Delta_{{\bf T}_1'} T_{2,j}^*T_{1,\alpha}^*h,y_j\right>\\
&=
\left<\varphi_{U^*}({\bf R})^*K_{{\bf T}_1}h, \left(\begin{matrix} 0\\ \vdots \\ e_\alpha\otimes y_j\\\vdots \\ 0\end{matrix}\right)\right>=
\left<\varphi_j({\bf R})^*K_{{\bf T}_1}h, e_\alpha\otimes y_j\right>.
\end{split}
\end{equation*}
Consequently,
$$
K_{{\bf T}_1'}T_{2,j}^*=\varphi_j({\bf R})^*K_{{\bf T}_1}, \qquad j=1,\ldots, n_2.
$$
   The last two relations in the theorem follow easily using the definition of the noncommutative Poisson kernes associated with   row contractions.
The proof is complete.
\end{proof}
We remark that, in Theorem \ref{dil},    if there is $U\in \cU_{\bf T}^\cK$ satisfying condition \eqref{cond-iso}, then $\varphi_{U^*}({\bf R})$ is  an isometry.
Here is an  example when condition \eqref{cond-iso} is satisfied. Another class of examples will be considered  later  in Theorem \ref{A-M2}.

Let $f=(f_1,\ldots, f_{n_1})$ be a free holomorphic automorphism of the unit ball $[B(\cH)^{n_1}]_1$. Then ${\bf T}_1:=(f_1({\bf S}),\ldots, f_{n_1}({\bf S}))$ is a pure row contraction and the noncommutative Poisson kernel $K_{{\bf T}_1}$ is a unitary operator on $F^2(H_{n_1})$ (see \cite{Po-automorphism}). Let ${\bf T}_2:=(\psi_1({\bf R}),\ldots \psi_{n_2}({\bf R}))$, with $\psi_j({\bf R})\in B(F^2(H_{n_1}))$,  be any multi-analytic  operator  which is  an isometry.
Applying Theorem \ref{dil}, when  $U\in \cU_{{\bf T}_1}^\cK$, we find  a multi-analytic operator
$\varphi_{U^*}({\bf R})=(\varphi_1({\bf R}),\ldots, \varphi_{n_2}({\bf R}))$ such that
$$
K_{{\bf T}_1} \psi_j({\bf R})^* =\varphi_j({\bf R})^*K_{{\bf T}_1}, \qquad j\in\{1,\ldots, n_2\}
$$
Since $K_{{\bf T}_1}$ is unitary, $\varphi_{U^*}({\bf R})$ is an isometry and, consequently, condition  \eqref{cond-iso} is satisfied.

It remains an open problem whether or not the set $\cU_{\bf T}^\cK$ always contains a unitary
operator $U$ such that $\varphi_{U^*}({\bf R})$ is an isometry ?

In what follows we recall a few facts concerning  constrained noncommutative Poisson kernels associated with noncommutative varieties in the unit ball $[B(\cH)^n]_1$ (see \cite{Po-varieties}).
Let $J$ be a WOT-closed two-sided ideal of $F_n^\infty$
such that $J\neq F_n^\infty$, and define the noncommutative variety $\cV_J(\cH)$ to be  the set of all pure row contractions ${\bf T}:=[T_1,\ldots, T_n]$, \ $T_i\in B(\cH)$, such that
$$
f(T_1,\ldots, T_n)=0\quad \text{ for any }\quad f\in J,
$$
where $f(T_1,\ldots, T_n)$ is defined using  the $F_n^\infty$-functional calculus for row contractions (see \cite{Po-funct}).
We proved in \cite{Po-varieties}  that there is a universal model  ${\bf B}:=[B_1,\ldots, B_n]$ associated with  $\cV_J(\cH)$, such that $f(B_1,\ldots, B_n)=0$ for any $f\in J$. More precisely, we showed that  an $n$-tuple
${\bf T}\in B(\cH)^n$ of operators   is in the noncommutative variety $\cV_J(\cH)$ if and only if $[T_1,\ldots, T_n]$ is unitarily equivalent to the compression  $[P_\cN (B_1\otimes I_\cD)|_{\cN}, \ldots, P_\cN (B_n\otimes I_\cD)|_{\cN}]$, where $\cD$ is some Hilbert space and  $\cN$ is
a co-invariant subspace of $B_1\otimes I_\cD, \ldots, B_n\otimes I_\cD$.

Let $J$ be a WOT-closed two-sided ideal of $F_n^\infty$
such that $J\neq F_n^\infty$, and define the subspaces of  the full Fock space $F^2(H_n)$ by setting
$$
\cM_J:=\overline{JF^2(H_n)}\quad \text{and}\quad \cN_J:=F^2(H_n)\ominus \cM_J.
$$
We proved in \cite{Po-varieties} that
  $\cN_J\neq \{0\}$ if and only if $ J\neq F_n^\infty$. One can easily see that the subspace $\cN_J$   is invariant under $S_1^*,\ldots, S_n^*$ and  $R_1^*,\ldots, R_n^*$.
 Define the {\it constrained  left} (resp.~{\it right}) {\it creation operators} by setting
$$B_i:=P_{\cN_J} S_i|\cN_J \quad \text{and}\quad W_i:=P_{\cN_J} R_i|\cN_J,\quad i=1,\ldots, n.
$$
and call ${\bf B}:=[B_1,\ldots, B_n]$ and ${\bf W}:=[W_1,\ldots, W_n]$  universal models  associated with the noncommutative variety $\cV_J$.

Let $F_n^\infty(\cV_J)$ be the $w^*$-closed algebra generated by
$B_1,\ldots, B_n$ and the identity. We proved in \cite{APo}
that  $F_n^\infty(\cV_J)$ has the $\mathbb A_1(1)$ property and therefore the $w^*$ and WOT topologies coincide on this algebra.
Similar results hold for $\cR_n^\infty(\cV_J)$, the $w^*$-closed algebra generated by
$W_1,\ldots, W_n$ and the identity. Moreover, we  proved in \cite{APo} that
$$
F_n^\infty(\cV_J)^\prime=\cR_n^\infty(\cV_J)\ \text{ and } \ \cR_n^\infty(\cV_J)^\prime=F_n^\infty(\cV_J),
$$
where $^\prime$ stands for the commutant.
An operator $M\in B(\cN_J\otimes \cK,\cN_J\otimes \cK')$ is called multi-analytic with respect to universal model ${\bf B}:=[B_1,\ldots, B_n]$
 if
$$
M(B_i\otimes I_{\cK})=(B_i\otimes I_{\cK'})M,\quad i=1,\ldots, n.
$$
We recall from \cite{Po-central} that the set of all multi-analytic  operators with respect to
 ${\bf B}$ coincides  with
$$
 \cR_n^\infty(\cV_J) \bar\otimes  B(\cK,\cK')=P_{\cN_J\otimes \cK'}[\cR_n^\infty\bar\otimes B(\cK,\cK')]|_{\cN_J\otimes \cK}.
$$
Similar results hold for  the algebra $F_n^\infty(\cV_J)$.

 We  remark that in the particular case when the ideal $J$ is generated by the polynomials
$S_iS_j-S_jS_i$, \ $i,j=1,\ldots, n$, then we have  $\cN_J=F^2_s$, the symmetric Fock space,    and $B_i:=P_{F_s^2} S_i|_{F_s^2}$, $i=1,\dots, n$, are the creation operators on the symmetric Fock space.
Arveson showed in  \cite{Arv2} that $F_s^2$  can be identified with his space $H^2$,
  of analytic functions in $\BB_n$, namely,
the
reproducing kernel Hilbert space
with reproducing kernel $K_n: \BB_n\times \BB_n\to \CC$ defined by
 $$
 K_n(z,w):= {\frac {1}
{1-\langle z, w\rangle_{\CC^n}}}, \qquad z,w\in \BB_n.
$$
 The algebra
$\cW_n^\infty:=P_{F_s^2} F_n^\infty|_{F_s^2}$,
    was    proved to be  the $w^*$-closed algebra generated by  the operators
   $B_i$, \ $i=1,\dots, n$, and the identity (see \cite{APo}).
     Moreover, Arveson showed in  \cite{Arv2}
   that $\cW_n^\infty$ can be identified with  the algebra of all  multipliers  of $H^2$.  Under this identification the creation operators $B_1,\ldots, B_n$ become the multiplication operators $M_{z_1},\ldots, M_{z_n}$ by the coordinate functions $z_1,\ldots, z_n$ of $\CC^n$.

Let $J\neq F_n^\infty$
be a WOT-closed two-sided ideal of $F_n^\infty$ and
let ${\bf T}:=[T_1,\ldots, T_n]$  be in the noncommutative variety $\cV_J(\cH)$.
The {\it constrained Poisson kernel} $K_{J,{\bf T}}:\cH\to \cN_J\otimes \cD_{\bf T}$ is defined by
$$K_{J,{\bf T}}:=(P_{\cN_J}\otimes I_{\cD_{\bf T}})K_{\bf T}
$$
is an isometry and satisfies the relation
\begin{equation*}
K_{J,{\bf T}}T_\alpha^*=(B_\alpha^*\otimes I_{\cD_{\bf T}})K_{J,{\bf T}},\quad  \alpha\in \FF_n^+.
\end{equation*}

The analogue of Theorem \ref{dil} on noncommutative varieties is the following. Recall that $\cU_{\bf T}^\cK$ is the set of all unitary extensions of the isometry defined by  relation \eqref{iso2}.

\begin{theorem} \label{dil-com}  Let
  ${\bf T}_1:=(T_{1,1},\ldots, T_{1,n_1})$ and  ${\bf T}_1':=(T_{1,1}',\ldots, T_{1,n_1}')$ be elements of the noncommutative varieties $\cV_J(\cH)$ and $\cV_J(\cH')$, respectively,  and let ${\bf T}_2:=(T_{2,1},\ldots, T_{2,n_2})$ be a row contraction such that ${\bf T}_2\in \cI({\bf T}_1,{\bf T}_1')$. Let ${\bf B}:=[B_{1,1},\ldots, B_{1, n_1}]$ and ${\bf W}:=[W_{1,1},\ldots, W_{1,n_1}]$ be the universal models associated with the vatiety $\cV_J$.
If $U\in \cU_{\bf T}^\cK$ and
$\varphi_{U^*}({\bf R})=(\varphi_1({\bf R}),\ldots, \varphi_{n_2}({\bf R}))$ is the associated  multi-analytic operator,    then the following relations hold:
$$
K_{J,{\bf T}_1'} T_{2,j}^* =\varphi_j({\bf W})^*K_{J,{\bf T}_1}, \qquad j\in\{1,\ldots, n_2\}
$$
and
$$
K_{J,{\bf T}_1}T_{1,i}^*=
\left(B_{1,i}^*\otimes I_{\cD_{{\bf T}_1}}\right)  K_{J,{\bf T}_1},\quad
K_{J,{\bf T}_1'}(T_{1,i}')^*=
\left(B_{1,i}^*\otimes I_{\cD_{{\bf T}_1'}}\right)  K_{J,{\bf T}_1'}, \qquad i\in \{1,\ldots, n_1\},
$$
 where  $K_{J,{\bf T}_1}$ and $K_{J,{\bf T}_1'}$ are the  constrained  Poisson kernels associated with ${\bf T}_1$ and ${\bf T}_1'$, respectively.

\end{theorem}
\begin{proof} Since ${\bf T}_1\in \cV_J(\cH)$ and ${\bf T}_1'\in \cV_J(\cH)$,
the noncommutative Poisson kernels $K_{{\bf T}_1}$ and $K_{{\bf T}_1'}$ have ranges in $\cN_J\otimes  \cD_{{\bf T}_1}$ and $\cN_J\otimes  \cD_{{\bf T}_1'}$, respectively.
Due to Theorem \ref{dil}, we have
\begin{equation}\label{INT}
K_{{\bf T}_1'} T_{2,j}^* =\varphi_j({\bf R})^*K_{{\bf T}_1}, \qquad j\in\{1,\ldots, n_2\},
\end{equation}
where ${\bf R}:=[R_{1,1},\ldots, R_{1, n_1}]$ is the tuple of right creation operators on the full Fock space $F^2(H_{n_1})$.
Since $\cN_J$ is co-invariant under $R_{1,1},\ldots, R_{1,n_1}$, we have
$$
\varphi_j({\bf R})^*(\cN_J\otimes  \cD_{{\bf T}_1})\subset \cN_J\otimes  \cD_{{\bf T}_1'}, \qquad j\in\{1,\ldots, n_2\},
$$
and
$$
 \varphi_j({\bf W})=P_{\cN_J\otimes  \cD_{{\bf T}_1}} \varphi_j({\bf R})|_{\cN_J\otimes  \cD_{{\bf T}_1'}}, \qquad j\in\{1,\ldots, n_2\},
 $$
 is a multi-analytic operator with respect to the universal model ${\bf B}$. Note that relation \eqref{INT} implies
 $$
 P_{\cN_J\otimes  \cD_{{\bf T}_1'}}K_{{\bf T}_1'} T_{2,j}^* =P_{\cN_J\otimes  \cD_{{\bf T}_1'}}\varphi_j({\bf R})^*P_{\cN_J\otimes  \cD_{{\bf T}_1}}K_{{\bf T}_1}, \qquad j\in\{1,\ldots, n_2\},
 $$
 which  proves that
 $$
 K_{J,{\bf T}_1'} T_{2,j}^* =\varphi_j({\bf W})^*K_{J,{\bf T}_1}, \qquad j\in\{1,\ldots, n_2\}.
 $$
 Since the other relations are known, the proof is complete.
 \end{proof}

As a consequence of Theorem \ref{dil}, we can obtain a new proof for the commutant lifting theorem for pure row contractions (\cite{S}, \cite{Po-isometric}) as well as a  new constructive method to obtain the lifting.

\begin{theorem} \label{CLT}
Let ${\bf T}=(T_1,\ldots, T_n)\in B(\cH)^n$ and  ${\bf T}'=(T_1',\ldots, T_n')\in B(\cH')^n$ be pure row contractions and let $A:\cH'\to \cH$ be an operator such that
$$AT_i'=T_iA, \qquad i\in\{1,\ldots, n\}.
 $$
 Let  ${\bf V}=(V_1,\ldots, V_n)\in B(\cK)^n$
and ${\bf V}'=(V_1',\ldots, V_n')\in B(\cK')^n$ be the minimal isometric dilations of ${\bf T}$  and ${\bf T}'$ on some Hilbert spaces $\cK\supset \cH$ and $\cK'\supset \cH'$, respectively.
 Then there is an operator $B:\cK'\to \cK$ such that
$$
BV_i'=V_iB, \qquad
i\in \{1,\ldots, n\},
$$
 $B^*|_\cH=A^*$, and $\|B\|=\|A\|$.
\end{theorem}
\begin{proof}
Since ${\bf T}$ is a pure row contraction, the noncommutative Poisson kernel $K_{\bf T}$ is an isometry,  and the minimal isometric dilation  of ${\bf T}$ is
${\bf V}:=(S_1\otimes I_{\cD_{{\bf T}}},\ldots S_n\otimes I_{\cD_{\bf T}})$, where $S_1,\ldots, S_n$ are the left creation operators on the full Fock space $F^2(H_n)$.  Under the identification of $\cH$ with $K_{\bf T}\cH$, we have
$$
T_i^*=(S_i^*\otimes I_{\cD_{\bf T}})|_\cH,\qquad \in \{1,\ldots, n\}.
$$
Similar assertions can be made about ${\bf T}'$ and its minimal isometric dilation ${\bf V}':=(S_1\otimes I_{\cD_{{\bf T}'}},\ldots S_n\otimes I_{\cD_{{\bf T}'}})$.

By scaling, we can assume that $\|A\|=1$, without loss of generality.
Since $A\in \cI({\bf T}', {\bf T})$, we can apply Theorem \ref{dil} in the particular case when $n_2=1$ and ${\bf T}_2:=A$. Consequently, there is a contractive multi-analytic operator $\varphi({\bf R})\in \cR_n^\infty\bar \otimes B(\cD_{{\bf T}_1'},\cD_{{\bf T}_1})$ such that
$K_{{\bf T}'} A^*=\varphi({\bf R})^* K_{\bf T}$. Under the identifications of $\cH$ with $K_{\bf T}\cH$ and $K_{{\bf T}'}\cH$,respectively,  we have
$A^*=\varphi({\bf R})^*|_\cH$.
Since $1=\|A^*\|\leq \|\varphi({\bf R})^*\|\leq 1$, we deduce that $\|A\|=\|\varphi({\bf R})\|$.
Since $B:=\varphi({\bf R})$ intertwines $S_i\otimes I_{\cD_{{\bf T}'}}$ with
$S_i\otimes I_{\cD_{{\bf T}}}$ for each $i\in \{1,\ldots, n\}$, the proof is complete.
\end{proof}

We remark that, as  a consequence of Theorem \ref{dil-com}, we can obtain  the commutant lifting theorem for pure   row contractions in noncommutative varieties (\cite{Po-varieties}).

\begin{theorem} \label{CLT2}

Let $J\neq F_n^\infty$ be a WOT-closed two-sided ideal of  the noncommutative analytic Toeplitz algebra $F_n^\infty$  and let ${\bf B}:=[B_1,\ldots, B_n]$  and ${\bf W}:=[W_1,\ldots, W_n]$ be the corresponding constrained shifts acting
on $\cN_J$.  For each $j=1,2$, let $\cK_j$ be a Hilbert space and $\cE_j\subseteq \cN_J\otimes \cK_j$ be a co-invariant subspace  under each operator $B_i\otimes I_{\cK_j}$, \ $i=1,\ldots, n$.
If $X:\cE_1\to \cE_2$ is a bounded operator such that
\begin{equation}\label{int}
X[P_{\cE_1}(B_i\otimes I_{\cK_1})|_{\cE_1}]=[P_{\cE_2}(B_i\otimes I_{\cK_2})]|_{\cE_2}X,\quad i=1,\ldots,n,
\end{equation}
then there exists
$$G({\bf W})\in \cR_n^\infty(\cV_J)\bar\otimes B(\cK_1,\cK_2)$$
such that
$$
G({\bf W})^*|_{\cE_2}=X^* \quad \text{ and }\quad \|G({\bf W})\|=\|X\|.
$$
\end{theorem}

 Another consequence of Theorem \ref{dil-com} is the following  Schur \cite{Sc} type representation   for the  unit ball of $\cR_n^\infty(\cV_J)\bar \otimes B(\cH', \cH)$.

\begin{theorem} \label{transfer} Let $J\neq F_n^\infty$ be a WOT-closed two-sided ideal of  the noncommutative analytic Toeplitz algebra $F_n^\infty$  and let ${\bf W}:=[W_1,\ldots, W_n]$ be the corresponding constrained  right creation  operators acting
on $\cN_J$.
 An operator $\Gamma:\cN_J\otimes \cH'\to \cN_J\otimes \cH$ is in the closed unit ball of $\cR_n^\infty(\cV_J)\bar \otimes B(\cH', \cH)$   if and only if there is a Hilbert space $\cE$ and a  unitary operator
$$
 \Lambda=\left[ \begin{matrix} E&F\\G&H \end{matrix}\right]: \begin{matrix} \cH'\\ \oplus \\ \cE \end{matrix}\to   \begin{matrix} \cH\\ \oplus \\ \bigoplus_{i=1}^n\cE \end{matrix}
$$
such that $\Gamma=\text{\rm SOT-}\lim_{r\to 1}\varphi_{\Lambda}(r{\bf W})$, where
\begin{equation*}
\begin{split}
\varphi_{\Lambda}(r{\bf W}):= I_{\cN_J}\otimes E+\left(I_{\cN_J}\otimes F\right)& \left(I_{\cN_J\otimes \cH}-\sum_{i=1}^{n}r W_{i}\otimes H_j\right)^{-1}\\
& \times \left[r W_{1}\otimes I_{\cH},\ldots,
rW_{n}\otimes I_{\cH}\right]\left(I_{\cN_J}\otimes G\right).
\end{split}
\end{equation*}
and  $H=\left[ \begin{matrix} H_1\\\vdots\\H_n\end{matrix} \right]:\cE\to \bigoplus_{i=1}^n \cE$.

\end{theorem}

\begin{proof} Assume that $\Gamma:\cN_J\otimes \cH'\to \cN_J\otimes \cH$ is
a contractive multi-analytic operator  with respect to the universal model ${\bf B}=[B_1, \ldots, B_n]$ of the variety $\cV_J$, i.e. $\Gamma(B_i\otimes I_{\cH'})=(B_i\otimes I_{\cH})\Gamma$ for any $i\in \{1,\ldots, n\}$. Due to the commutant lifting theorem for pure row contractions, Theorem \ref{CLT}, there exists
 a contractive multi-analytic  operator $\Psi:F^2(H_n)\otimes \cH'\to F^2(H_n)\otimes \cH$  with respect to the universal model ${\bf S}$, i.e.
 $\Psi(S_i\otimes I_{\cH'})=(S_i\otimes I_{\cH})\Psi$ for any $i\in \{1,\ldots, n\}$,  such that $\|\Gamma\|=\|\Psi\|$ and $\Psi^*|_{\cN_J\otimes \cH}=\Gamma^*$.

 Set ${\bf T}_1:=(S_1\otimes I_\cH,\ldots S_n\otimes I_{\cH})$,
${\bf T}_1':=(S_1\otimes I_{\cH'},\ldots S_n\otimes I_{\cH'})$, and
${\bf T}_2:=\Psi$. Since $\Psi\in \cI({\bf T}, {\bf T}')$,  Theorem \ref{dil} shows that there is a unitary operator
$$
 \Lambda=\left[ \begin{matrix} E&F\\G&H \end{matrix}\right]: \begin{matrix} \cD_{{\bf T}_1'}\\ \oplus \\ \cE \end{matrix}\to   \begin{matrix} \cD_{{\bf T}_1}\\ \oplus \\ \bigoplus_{i=1}^n\cE \end{matrix}
$$
such that $\varphi_{\Lambda}({\bf R}):=\text{\rm SOT-}\lim_{r\to 1}\varphi_{\Lambda}(r{\bf R})$ is a multi-analytic operator in $\cR_{n}^\infty\bar\otimes B(\cD_{{\bf T}_1'},\cD_{{\bf T}_1})$, where $\varphi_{\Lambda}(r{\bf R})$ is  defined as in the theorem and such that
$K_{{\bf T}_1} \Psi^*=\varphi_{\Lambda}({\bf R})^*K_{{\bf T}_1}$.
Note that  $\cD_{{\bf T}_1}=\cH$, $\cD_{{\bf T}_1'}=\cH'$, and the noncommutative Poisson kernel $K_{{\bf T}_1}$ is the identity on
$F^2(H_n)\otimes \cH$. Consequently, $\Psi=\varphi_{\Lambda}({\bf R})$.
Since $\Psi^*|_{\cN_J\otimes \cH}=\Gamma^*$, we deduce that $\Gamma=\varphi_\Lambda({\bf W})$.

 To prove the converse, note that,  in the particular case when $n_2=1$ and $U^*=\Lambda$,   Lemma \ref{strong-limit} shows that
 $\Psi=\text{\rm SOT-}\lim_{r\to 1}\varphi_{\Lambda}(r{\bf R})$, where
\begin{equation*}
\begin{split}
\varphi_{\Lambda}(r{\bf R}):= I_{F^2(H_{n})}\otimes E+\left(I_{F^2(H_{n})}\otimes F\right)& \left(I_{F^2(H_{n})\otimes \cH}-\sum_{i=1}^{n}r R_{i}\otimes H_j\right)^{-1}\\
& \times \left[r R_{1}\otimes I_{\cH},\ldots,
rR_{n}\otimes I_{\cH}\right]\left(I_{F^2(H_{n})}\otimes G\right).
\end{split}
\end{equation*}
and  $H=\left[ \begin{matrix} H_1\\\vdots\\H_n\end{matrix} \right]:\cE\to \bigoplus_{i=1}^n \cE$, is a contractive multi-analytic operator with respect to ${\bf S}$.
 Since $\cN_J$ is a  co-invariant subspace under $R_1,\ldots, R_n$, we deduce that $\Gamma:=\varphi_\Lambda({\bf W})$ is a contractive multi-analytic operator with respect to ${\bf B}$.
 The proof is complete.
\end{proof}

As a particular case of Theorem \ref{transfer},
we obtain  the following  representation theorem for the unit ball of the noncommutative Hardy algebra $\cR_n^\infty(\cV_J)$ associated with the variety $\cV_J$.

\begin{corollary} \label{transfer2}  An operator $\Gamma\in B(\cN_J)$ is in the closed unit ball of $\cR_n^\infty(\cV_J)$   if and only if there is a Hilbert space $\cH$ and a  unitary operator
$$
 \Lambda=\left[ \begin{matrix} E&F\\G&H \end{matrix}\right]: \begin{matrix} \CC\\ \oplus \\ \cH \end{matrix}\to   \begin{matrix} \CC\\ \oplus \\ \bigoplus_{i=1}^n\cH \end{matrix}
$$
such that $\Gamma=\text{\rm SOT-}\lim_{r\to 1}\varphi_{\Lambda}({\bf W})$, where
\begin{equation*}
\begin{split}
\varphi_{\Lambda}(r{\bf W}):= I_{F^2(H_{n})}\otimes E+\left(I_{F^2(H_{n})}\otimes F\right)& \left(I_{F^2(H_{n})\otimes \cH}-\sum_{i=1}^{n}r W_{i}\otimes H_j\right)^{-1}\\
& \times \left[r W_{1}\otimes I_{\cH},\ldots,
rW_{n}\otimes I_{\cH}\right]\left(I_{F^2(H_{n})}\otimes G\right).
\end{split}
\end{equation*}
and  $H=\left[ \begin{matrix} H_1\\\vdots\\H_n\end{matrix} \right]:\cH\to \bigoplus_{i=1}^n \cH$.
\end{corollary}

\bigskip

\section{And\^ o type inequalities on  noncommutative varieties}

In this section, we obtain And\^ o type dilations and inequalities  for the elements of the bi-ball ${\bf P}_{\bf n}^-$ and a class of noncommutative varieties. Moreover, we prove that the bi-ball has a universal model consisting of the left creation operators and contractive multi-analytic operators. A similar result holds on noncommutative varieties.

One of the most important consequences of the results from Section 1 is the following And\^ o type dilation for noncommutative varieties in the bi-ball ${\bf P}_{\bf n}^-$. We recall that $\cU_{\bf T}^\cK$ is the set of all unitary extensions of the isometry defined by relation \eqref{iso2} (see also \eqref{ABCD}).  According to Lemma \ref{strong-limit}, for each $U\in \cU_{\bf T}^\cK$,
 the strong operator topology limit
$\varphi_{U^*}({\bf R}):=\text{\rm SOT-}\lim_{r\to 1}\varphi_{U^*}(r{\bf R})
$
 exists and defines a contractive  multi-analytic  operator.

\begin{theorem} \label{dil2}
Let ${\bf T}=({\bf T}_1, {\bf T}_2)\in {\bf P}_{\bf n}^-(\cH)$ with ${\bf T}_1=(T_{1,1},\ldots, T_{1,n_1})$ and ${\bf T}_2:=(T_{2,1},\ldots, T_{2,n_2})$.     If ${\bf T}_1 $ is in the noncommutative variety $  \cV_J(\cH)$ and $\varphi_{U^*}({\bf W})=(\varphi_1({\bf W}),\ldots, \varphi_{n_2}({\bf W}))$ is the  multi-analytic operator associated with $U\in \cU_{\bf T}^\cK$,
then
$$
K_{J,{\bf T}_1} T_{1,\alpha}^* T_{2,\beta}^*=\left(B_{1,\alpha}^*\otimes I_{\cD_{{\bf T}_1}}\right)\varphi_\beta({\bf W})^* K_{J,{\bf T}_1}, \qquad \alpha\in \FF_{n_1}^+, \beta\in \FF_{n_2}^+,
$$
  where ${\bf B}:=[B_{1,1},\ldots, B_{1,n_1}]$ and ${\bf W}:=[W_{1,1},\ldots, W_{1,n_1}]$ are the universal models of $\cV_J$ and $K_{J,{\bf T}_1}$ is  the associated  constrained  Poisson kernel.
\end{theorem}
\begin{proof} In the particular case when ${\bf T}_1={\bf T}_1'$,  Theorem \ref{dil-com} shows that
$$
K_{J,{\bf T}_1} T_{2,j}^* =\varphi_j({\bf W})^*K_{J{\bf T}_1}, \qquad j\in\{1,\ldots, n_2\}
$$
and
$$
K_{J,{\bf T}_1}T_{1,i}^*=
\left(B_{1,i}^*\otimes I_{\cD_{{\bf T}_1}}\right)  K_{J,{\bf T}_1},  \qquad i\in \{1,\ldots, n_1\}.
$$
Hence, the relation in the theorem follows.
\end{proof}

We remark that Theorem \ref{dil2} provides a model  and a characterization of the elements $({\bf T}_1, {\bf T}_2)\in {\bf P}_{\bf n}^-(\cH)$ with ${\bf T}_1\in \cV_J(\cH)$ .
Indeed, if  ${\bf T}=({\bf T}_1, {\bf T}_2)\in B(\cH)^{n_1}\times_c B(\cH)^{n_2}$, then ${\bf T}\in {\bf P}_{\bf n}^-(\cH)$ with ${\bf T}_1\in \cV_J(\cH)$ if and only if there is a Hilbert space $\cD$, a contractive multi-analytic operator $(\varphi_1({\bf W}),\ldots, \varphi_{n_2}({\bf W})$
with respect to ${\bf B}$, $\varphi_j({\bf W})\in B(\cN_J\otimes \cD)$, and a co-invariant subspace $\cM\subset \cN_J\otimes \cD$ under each of the operators
$B_{1,i}\otimes I_\cD$ and  $\varphi_j({\bf W})$, where $i\in \{1,\ldots, n_1\}$ and
$j\in \{1,\ldots, n_2\}$, such that $\cM$ can be identified with $\cH$,
$$
(B_{1,i}^*\otimes I_\cD)|_\cH=T_{1,i}^*, \quad \text{and} \quad \varphi_j({\bf W})^*|_\cH=T_{2,j}^*.
$$
Note that the direct implication  is due to Theorem \ref{dil2} under the identification of $\cH$ with $K_{J,{\bf T}_1}\cH$. The converse is obvious.

In what follows, we obtain And\^ o type inequalities  for the bi-ball ${\bf P}_{\bf n}^-(\cH)$ and the noncommutative variety
$${{\bf P}_{c,{\bf n}}^-}(\cH):=\left\{({\bf T}_1, {\bf T}_2)\in {\bf P}_{\bf n}^-(\cH): {\bf T}_1\in \cV_J(\cH)\right\}.
$$
First, we consider the case when ${\bf T}_1=(T_{1,1},\ldots, T_{1,n_1})$ and ${\bf T}_2:=(T_{2,1},\ldots, T_{2,n_2})$ are  row contractions on a Hilbert space with the property that ${\bf T}=({\bf T}_1, {\bf T}_2)\in {\bf P}_{c,{\bf n}}^-(\cH)$ with
$d_i:=\dim \cD_{{\bf T}_i}<\infty$   and $d_1+n_1d_2=d_2+n_2d_1$.
We recall that $\cU_{\bf T}$  is the set of all  unitary extensions
 $U:\cD_{{\bf T}_1}\oplus \bigoplus_{j=1}^{n_1}\cD_{{\bf T}_2} \to
\bigoplus_{j=1}^{n_2}\cD_{{\bf T}_1}\oplus \cD_{{\bf T}_2}$ of the isometry
\begin{equation}\label{UU}
U\left(\Delta_{{\bf T}_1}h, \Delta_{{\bf T}_2}T_{1,1}^*h, \ldots, \Delta_{{\bf T}_2}T_{1,n_1}^*h\right):=
\left(\Delta_{{\bf T}_1}T_{2,1}^*h, \ldots, \Delta_{{\bf T}_1}T_{2,n_2}^*h,
\Delta_{{\bf T}_2}h\right),\qquad h\in \cH.
\end{equation}

Given $n_1,n_2\in \NN$,  let ${\bf X}:=\left< X_1,\ldots, X_{n_1}\right>$ and
${\bf Y}:=\left< Y_1,\ldots, Y_{n_2}\right>$  be noncommutative indeterminates and assume that
 $X_iY_j=Y_jX_i$ for any $i\in \{1,\ldots, n_1\}$ and $j\in \{1,\ldots, n_2\}$. We denote by  $\CC\left<{\bf X}, {\bf Y}\right>$  the complex algebra of all  polynomials in    $X_{1},\ldots, X_{n_1}$ and $Y_{1},\ldots, Y_{n_2}$. Note that when $n_1=n_2=1$, then $\CC\left<{\bf X}, {\bf Y}\right>$ coincides with the algebra $\CC[z,w]$ of complex polynomials in two variable.

\begin{theorem}\label{ando}
Let ${\bf T}_1=(T_{1,1},\ldots, T_{1,n_1})$ and ${\bf T}_2:=(T_{2,1},\ldots, T_{2,n_2})$ be   row contractions on a Hilbert space such that
$$d_i:=\dim \cD_{{\bf T}_i}<\infty\ \text{ and } \ d_1+n_1d_2=d_2+n_2d_1=m,
$$
and  let ${\bf B}:=[B_{1,1},\ldots, B_{1,n_1}]$ and ${\bf W}:=[W_{1,1},\ldots, W_{1,n_1}]$ be  the universal models of the noncommutative variety $\cV_J$.
If ${\bf T}=({\bf T}_1, {\bf T}_2)\in {\bf P}_{c,{\bf n}}^-(\cH)$  and $\Lambda\in \cU_{\bf T}$,   then
$$
\|[p_{rs}({\bf T}_1,{\bf T}_2)]_{k}\|\leq   \|[p_{rs}({\bf B}\otimes I_{\CC^{d_1}}, \varphi_\Lambda ({\bf  W}))]_{k}\|, \qquad [p_{rs}]_{ k}\in M_k(\CC\left<{\bf X}, {\bf Y}\right>), k\in \NN,
$$
where  $\varphi_\Lambda ({\bf W})=(\varphi_1({\bf W}),\ldots,\varphi_{n_2}({\bf W}))$
is a contractive  operator uniquely determined by $\Lambda$ and each $\varphi_j({\bf W}) $ is a $d_1\times d_1$-matrix with entries in the Hardy algebra
 $\cR_{n_1}^\infty(\cV_J)$.
 \end{theorem}

 \begin{proof} Since $d_1+n_1d_2=d_2+n_2d_1=m$, the set $\cU_{\bf T}$ of all unitary extensions of the isometry $U$ defined by relation \eqref{UU} is non-empty. Fix any $\Lambda\in\cU_{\bf T}$ and apply Theorem \ref{dil2} to ${\bf T}=({\bf T}_1, {\bf T}_2)\in {\bf P}_{c,{\bf n}}^-(\cH)$  and $\Lambda\in \cU_{\bf T}$.
  Then the contractive  multi-analytic operator $\varphi_\Lambda ({\bf W})=(\varphi_1({\bf W}),\ldots,\varphi_{n_2}({\bf W}))$ with $\varphi_j({\bf W})\in \cR_{n_1}^\infty(\cV_J)\bar\otimes \cD_{{\bf T}_1}$ has the property that
 $$
K_{J,{\bf T}_1} T_{1,\alpha}^* T_{2,\beta}^*=\left(B_{1,\alpha}^*\otimes I_{\cD_{{\bf T}_1}}\right)\varphi_\beta({\bf W})^* K_{J,{\bf T}_1}
$$
for any $\alpha\in \FF_{n_1}^+$ and $\beta\in \FF_{n_2}^+$. Consequently, if $p$ is any polynomial in $\CC\left<{\bf X}, {\bf Y}\right>$, we obtain
$$
K_{J,{\bf T}_1} p({\bf T}_1,{\bf T}_2)=p({\bf B}\otimes I_{\CC^{d_1}}, \varphi_\Lambda ({\bf  W}))K_{J,{\bf T}_1}.
$$
Since ${\bf T}_1\in \cV_J(\cH)$, the noncommutative Poisson kernel $K_{J,{\bf T}_1}$ is an isometry, which implies
$$
 p({\bf T}_1,{\bf T}_2)=K_{J,{\bf T}_1}^*p({\bf B}\otimes I_{\CC^{d_1}}, \varphi_\Lambda ({\bf  W}))K_{J,{\bf T}_1}.
$$
Now, it is clear that
$$
\|[p_{rs}({\bf T}_1,{\bf T}_2)]_{k\times k}\|\leq   \|[p_{rs}({\bf B}\otimes I_{\CC^{d_1}}, \varphi_\Lambda ({\bf  W}))]_{k\times k}\|, \qquad [p_{rs}]_{k\times k}\in M_k(\CC\left<{\bf X}, {\bf Y}\right>), k\in \NN,
$$
 The proof is complete.
 \end{proof}

Denote by $\cQ_{\bf n}^*$ the set of all formal polynomials of the form
$q({\bf X}, {\bf Y})=\sum a_{\alpha,\beta, \gamma, \sigma} X_\alpha Y_\beta Y_\sigma^*X_\gamma^* $, with complex coefficients, where ${\bf X}:=\left< X_1,\ldots, X_{n_1}\right>$ and
${\bf Y}:=\left< Y_1,\ldots, Y_{n_2}\right>$. In what follows, we show that
if we drop the conditions $d_i:=\dim \cD_{{\bf T}_i}<\infty$ and $d_1+n_1d_2=d_2+n_2d_1=m$, in Theorem \ref{ando},  we can obtain the following And\^ o type inequality.

\begin{theorem} \label{ando1}
   Let ${\bf T}_1=(T_{1,1},\ldots, T_{1,n_1})$ and ${\bf T}_2:=(T_{2,1},\ldots, T_{2,n_2})$ be  row contractions with  the property that   ${\bf T}=({\bf T}_1, {\bf T}_2)\in {\bf P}_{c,{\bf n}}^-(\cH)$.  If $U\in \cU_{{\bf T}}^\cK$, then
$$
\|[q_{rs}({\bf T}_1,{\bf T}_2)]_{k}\|\leq   \|[q_{rs}({\bf V}_1,{\bf V}_2)]_{k}\|, \qquad [p_{rs}]_{k\times k}\in M_k(\cQ_{\bf n}^*),
$$
where ${\bf V}_1:={\bf B}\otimes I_{\cD_{{\bf T}_1}}$ and  ${\bf V}_2:= \varphi_{U^*}({\bf  W})$
 is the   multi-analytic operator associated with $U^*$.
\end{theorem}
\begin{proof} The proof uses Theorem \ref{dil2} and is similar to the proof of Theorem \ref{ando}. We shall omit it.
\end{proof}

\begin{lemma}\label{norm}
For any  polynomial $p\in \CC\left<{\bf X},{\bf Y}\right>$,  define
$$
\|p\|_u:=\sup \|p({\bf T}_1, {\bf T}_2)\|,
$$
where the supremum is taken over all pairs
$({\bf T}_1, {\bf T}_2)\in {\bf P}_{\bf n}^-(\cH)$ and any Hilbert space $\cH$.
Then  $\|\cdot\|_u$ defines an algebra  norm on $\CC\left<{\bf X},{\bf Y}\right>$.
\end{lemma}

 \begin{proof}
 The only non  trivial fact is that $\|p\|_u=0$ implies $p=0$.
Let $p=\sum_{|\alpha|\leq m, |\beta|\leq m} a_{\alpha, \beta} X_\alpha Y_\beta$, where $\alpha\in \FF_{n_1}^+, \beta\in \FF_{n_2}^+$,  and assume that $\|p\|_u=0$. Without loss of generality we can assume that $n_2\leq n_1$ and $\FF_{n_2}^+\subset \FF_{n_1}^+$.
If $z_i, w_j \in \DD$ for any $i\in \{1,\ldots, n_1\}$ and $j\in \{1,\ldots, n_2\}$, then
$$
(z_1S_1,\ldots, z_{n_1}S_{n_1}, w_1R_1,\ldots, w_{n_2} R_{n_2})\in {\bf P}_{\bf n}^-(F^2(H_{n_1})),
$$
where $S_1,\ldots, S_{n_1}$ and $R_1,\ldots,  R_{n_2}$ are creation operators on the full Fock space $F^2(H_{n_1})$.
Fix $\alpha_0\in \FF_{n_1}^+, \beta_0\in \FF_{n_2}^+$ with $|\alpha_0|\leq m $ and  $|\beta_0|\leq m$ and let $ k:=|\alpha_0|+|\beta_0|$.
Since $\|p\|_u=0$, we have
$\sum_{|\alpha|\leq m, |\beta|\leq m} a_{\alpha, \beta}z_\alpha w_\beta S_\alpha R_\beta=0$ for any $z_i, w_j \in \DD$. Consequently,
$$
\left<\sum_{|\alpha|\leq m, |\beta|\leq m} a_{\alpha, \beta}z_\alpha w_\beta S_\alpha R_\beta (1), e_{\alpha_0\tilde \beta_0}\right>=0,
$$
where $\widetilde\sigma$ is the reverse of $\sigma\in \FF_{n_1}^+$,
which implies $\sum a_{\alpha, \beta} z_\alpha w_\beta=0$, where the sum is taken over all $|\alpha|\leq m, |\beta|\leq m$  such that $\alpha \tilde\beta=\alpha_0\tilde\beta_0$.
Consequently, if we assume that $\alpha_0\tilde\beta_0=g_{i_1}\cdots g_{i_k}$, the latter equality becomes
\begin{equation}  \label{SS}
\begin{split}
a_{g_0, g_{i_k}\cdots g_{i_1}} w_{i_1}\cdots w_{i_k}&+
a_{g_{i_1}, g_{i_k}\cdots g_{i_2}} z_{i_1}w_{i_2}\cdots w_{i_k}\\
&+ a_{g_{i_1}g_{i_2}, g_{i_k}\cdots g_{i_3}} z_{i_1} z_{i_2} w_{i_3}\cdots w_{i_k}
+\cdots +a_{g_{i_1}\cdots g_{i_k},g_0} z_{i_1}\cdots z_{i_k}=0
\end{split}
\end{equation}
for any $z_i, w_j \in \DD$.
Taking $z_{i_1}=0$ in relation \eqref{SS}, we get $a_{g_0, g_{i_k}\cdots g_{i_1}}=0$. Dividing the resulting equality by $z_{i_1}\neq 0$ and taking $z_{i_2}=0$, we obtain $a_{g_{i_1}, g_{i_k}\cdots g_{i_2}}=0$. Continuing this process, we deduce that $a_{\alpha, \beta}=0$ for any $\alpha\in \FF_{n_1}^+, \beta\in \FF_{n_2}^+$ with
 $|\alpha|\leq m, |\beta|\leq m$  and  such that $\alpha \tilde\beta=\alpha_0\tilde\beta_0$. In particular, we have $a_{\alpha_0, \beta_0}=0$. This completes the proof.
\end{proof}

As in Lemma \ref{norm}, if for $[p_{ij}]\in M_k(\CC\left<{\bf X},{\bf Y}\right>)$, we set
$$
\|[p_{ij}]\|_{u,k}:=\sup \|[p_{ij}({\bf T}_1, {\bf T}_2)]\|,
$$
where the supremum is taken over all pairs
$({\bf T}_1, {\bf T}_2)\in {\bf P}_{\bf n}^-(\cH)$ and any Hilbert space $\cH$, we obtain a sequence of norms on the matrices over $\CC\left<{\bf X},{\bf Y}\right>$. We call $\left(\CC\left<{\bf X},{\bf Y}\right>, \|\cdot\|_{u,k}\right)$ the universal operator algebra for the bi-ball ${\bf P}_{\bf n}^-$.

In what follows, we  prove that the abstract bi-ball
${\bf P}_{\bf n}^-:=\{{\bf P}_{\bf n}^-(\cH): \ \cH \text{ is a Hilbert space }\}$, where ${\bf n}=(n_1,n_2)\in \NN^2$,
has a universal model $$({\bf S}\otimes I_{\ell^2},  \varphi({\bf  R}))\in
{\bf P}_{\bf n}^-(F^2(H_{n_1})\otimes \ell^2),
 $$
 where ${\bf S}=(S_1,\ldots, S_{n_1})$ and ${\bf R}=(R_1,\ldots, R_{n_1})$ are the left and right creation operators on the full Fock space $F^2(H_{n_1})$, respectively, and
$\varphi({\bf  R})=(\varphi_1({\bf R}),\ldots, \varphi_{n_2}({\bf R})$ is a contractive multi-analytic operator with $\varphi_j({\bf  R})\in B(F^2(H_{n_1})\otimes \ell^2)$.
The next result shows that $\left(\CC\left<{\bf X},{\bf Y}\right>, \|\cdot\|_{u,k}\right)$ can be realized completely isometrically isomorphic as a concrete algebra of operators.

\begin{theorem} \label{ando11}  There is a contractive   operator $\varphi({\bf  R})=(\varphi_1({\bf  R}),\ldots, \varphi_{n_2}({\bf  R}))$, where each $\varphi_j({\bf  R})$ is a multi-analytic operator acting on $F^2(H_n)\otimes \ell^2$,
such that
$$
\|[p_{rs}({\bf T}_1,{\bf T}_2)]_{k}\|\leq  \|[p_{rs}({\bf S}\otimes I_{\ell^2},  \varphi({\bf  R}))]_{k}\|, \qquad  p_{rs}\in \CC\left<{\bf X}, {\bf Y}\right>,
$$
 for any  $({\bf T}_1, {\bf T}_2)\in {{\bf P}_{{\bf n}}^-}(\cH)$ and  $k\in \NN$.
\end{theorem}

\begin{proof}
Given a polynomial $p\in \CC\left<{\bf X},{\bf Y}\right>$, we define
$
\|p\|_u:=\sup \|p({\bf T}_1, {\bf T}_2)\|,
$
where the supremum is taken over all pairs
$({\bf T}_1, {\bf T}_2)\in {\bf P}_{\bf n}^-(\cH)$ and any Hilbert space $\cH$. Using a standard argument, one can prove that the supremum is the same if we consider  only infinite dimensional separable Hilbert spaces.  Since $r{\bf T}$ is a pure row contraction for any $r\in [0,1)$, it is clear that
$$
\|p\|_u=\sup_{({\bf T}_1, {\bf T}_2)\in {\bf P}_{\bf n}, {\bf T}_1 \text{ pure}} \|p({\bf T}_1, {\bf T}_2)\|.
$$
Fix $[p_{ij}]_k\in M_k(\CC\left<{\bf X}, {\bf Y}\right>)$ and  choose a sequence $\left\{({\bf T}_1^{(m)}, {\bf T}_2^{(m)})\right\}_{m=1}^\infty$  in ${\bf P}_{\bf n}^-(\cH) $ with $\cH$ separable and ${\bf T}_1^{(m)}$ pure row contraction  such that
\begin{equation}
\label{sup2}
\|[p_{ij}]_k\|_u=\sup_{m}\|[p_{ij}({\bf T}_1^{(m)}, {\bf T}_2^{(m)})]_k\|.
\end{equation}
 Using Theorem \ref{ando1}, in the particular case when $J=\{0\}$,  we find,  for each $m\in \NN$,  a contractive operator $\varphi^{(m)}({\bf  R}):=(\varphi_1^{(m)}({\bf  R}),\ldots, \varphi_{n_2}^{(m)}({\bf  R}))$, where each $\varphi_j^{(m)}({\bf  R})$ is a multi-analytic operator acting on $F^2(H_{n_1})\otimes \CC^{d(m)}$, where $d(m):=\cD_{{\bf T}_1^{(m)}}$, such that
$$
\|[p_{ij}({\bf T}_1^{(m)},{\bf T}_2^{(m)})]_k\|\leq  \|[p_{ij}({\bf S}\otimes I_{\CC^{d(m)}},\varphi^{(m)}({\bf R}))]_{k }\|.
$$
Consequently, relation \eqref{sup2} implies
\begin{equation*}
\begin{split}
\|[p_{ij}]_k\|_u&=\|[p_{ij}(\oplus_{m=1}^\infty {\bf T}_1^{(m)}, \oplus_{m=1}^\infty {\bf T}_2^{(m)})]_k\|\\
&\leq
\|[p_{ij}(\oplus_{m=1}^\infty  ({\bf S}\otimes I_{\CC^{d(m)}}), \oplus_{m=1}^\infty \varphi^{(m)}({\bf R}))]_k\|\leq \|[p_{ij}]_k\|_u,
\end{split}
\end{equation*}
where $\oplus_{m=1}^\infty {\bf T}_1^{(m)}$ denotes the row contraction
$\left[\oplus_{m=1}^\infty {T}_{1,1}^{(m)},\ldots,\oplus_{m=1}^\infty { T}_{1,n_1}^{(m)}\right]$.
This shows  that
\begin{equation}
\label{ma22}
\|[p_{ij}]_k\|_u= \|[p_{ij}({\bf S}\otimes I_{ \ell^2},\psi({\bf R}))]_{k }\|,
\end{equation}
where $\psi({\bf R})=(\psi_1({\bf R}),\ldots, \psi_{n_2}({\bf R})):=\oplus_{m=1}^\infty \varphi^{(m)}({\bf R}))$ is
a contractive   operator and  $\psi_j({\bf R})$ is  a multi-analytic operator on the Fock  space $F^2(H_{n_1})\otimes \ell^2$.
Let $\CC_\QQ\left<{\bf X}, {\bf Y}\right>$ be the set of all polynomials with coefficients in $\QQ+i\QQ$, and let
$
[p^{(1)}_{ij}]_k, [p^{(2)}_{ij}]_k, \ldots$ be an enumeration of  the set $\{[p_{ij}]_k: \ p_{ij}\in \CC_\QQ\left<{\bf X}, {\bf Y}\right>\}$.
Due to relation \eqref{ma22}, for each $s\in \NN$, there is a contractive  operator  $\psi^{(s)}({\bf R})=(\psi^{(s)}_1({\bf R}),\ldots, \psi^{(s)}_{n_2}({\bf R}))$, where  $\psi^{(s)}({\bf R})$ is a  multi-analytic operator  on $F^2(H_{n_1})\otimes \ell^2$, such that
\begin{equation}
\label{ma222}
\|[p^{(s)}_{ij}]_k\|_u= \|[p_{ij}^{(s)}({\bf S}\otimes I_{ \ell^2},\psi^{(s)}({\bf R}))]_{k }\|,\qquad s\in \NN.
\end{equation}
Define the contractive multi-analytic operator $\Psi_k({\bf R}):=\oplus_{s=1}^\infty \psi^{(s)}({\bf R})$  and let us prove that
\begin{equation}
\label{qqij}
\|[q_{ij}]_k\|_u= \|[q_{ij}({\bf S}\otimes I_{ \ell^2},\Psi_k({\bf R}))]_{k }\|
\end{equation}
for any $[q_{ij}]_k\in M_k(\CC\left<{\bf X}, {\bf Y}\right>)$.
Note that relation \eqref{ma22} implies
\begin{equation}
\label{Ga2}
\|[p^{(s)}_{ij}]_k\|_u= \|[p_{ij}^{(s)}({\bf S}\otimes I_{ \ell^2},\Psi_k({\bf R}))]_{k }\|\qquad \text{for any } s\in \NN.
\end{equation}
 Fix $[q_{ij}]_k\in M_k(\CC\left<{\bf X}, {\bf Y}\right>)$ and $\epsilon >0$,  and choose $[p^{(s_0)}_{ij}]_k$ such that
 \begin{equation}\label{u2}
 \left\|[q_{ij}]_k-[p^{(s_0)}_{ij}]_k\right\|_u<\epsilon.
 \end{equation}
 Using relations \eqref{ma22}, \eqref{u2}, and \eqref{Ga2}, we deduce that there is a contractive multi-analytic operator $\psi^{(q)}({\bf R})$ on the Fock space $F^2(H_{n_1})\otimes \ell^2$ such that
 \begin{equation*}
 \begin{split}
 \|[q_{ij}]_k\|_u&=\|[q_{ij}({\bf S}\otimes I_{ \ell^2},\psi^{(q)}({\bf R}))]_{k }\|\\
 &
\leq\|[p_{ij}^{(s_0)}({\bf S}\otimes I_{ \ell^2},\psi^{(q)}({\bf R}))]_{k }\| +\epsilon\\
 &\leq \|[p_{ij}^{(s_0)}]_k\|_u+\epsilon\\
 &=
 \|[p_{ij}^{(s_0)}({\bf S}\otimes I_{ \ell^2},\Psi_k({\bf R}))]_{k }\|+\epsilon\\
 &\leq
 \|[q_{ij}({\bf S}\otimes I_{ \ell^2},\Psi_k({\bf R}))]_{k }\|+2\epsilon
 \end{split}
 \end{equation*}
for any $\epsilon>0$, which proves  relation \eqref{qqij}.
Note  that  $\varphi({\bf R}):=\oplus_{k=1}^\infty \Psi_k({\bf R})$ is
 a contractive multi-analytic operator  and
\begin{equation*}
\|[q_{ij}]_k\|_u= \|[q_{ij}({\bf S}\otimes I_{ \ell^2},\varphi({\bf R}))]_{k }\|
\end{equation*}
for any $[q_{ij}]_k\in M_k(\CC\left<{\bf X}, {\bf Y}\right>)$ and any $k\in \NN$.
The proof is complete.

\end{proof}

 The closed non-self-adjoint algebra generated by
$ S_1\otimes I_{\ell^2},\ldots, S_{n_1}\otimes I_{\ell^2}, \varphi_1({\bf R}),\ldots, \varphi_{n_2}({\bf R})$ and the identity is denoted by
$\cA({\bf P_n})$. In light of Theorem \ref{ando11} ,  $\cA({\bf P_n})$ can be seen as the universal operator algebra of the  bi-ball.
An important open question remains. Can one choose  $(\varphi_1({\bf R}),\ldots, \varphi_{n_2}({\bf R}))$ to be  an  isometry$?$ In the next section, we will see that the answer is positive when $n_1=n_2=1$.

We remark that the noncommutative variety
$${{\bf P}_{c,{\bf n}}^-}(\cH):=\left\{({\bf T}_1, {\bf T}_2)\in {\bf P}_{\bf n}^-(\cH): {\bf T}_1\in \cV_J(\cH)\right\}
$$
 also has a universal model. Similarly to the proof of Theorem \ref{ando11}, one can show that  there is a contractive   operator $\varphi({\bf  W})=(\varphi_1({\bf  W}),\ldots, \varphi_{n_2}({\bf  W}))$, where each $\varphi_j({\bf  W})$ is a multi-analytic operator, with respect to ${\bf B}$, acting on $\cN_J\otimes \ell^2$,
such that
$$
\|[p_{rs}({\bf T}_1,{\bf T}_2)]_{k}\|\leq  \|[p_{rs}({\bf B}\otimes I_{\ell^2},  \varphi({\bf  W}))]_{k}\|, \qquad  p_{rs}\in \CC\left<{\bf X}, {\bf Y}\right>,
$$
 for any  $({\bf T}_1, {\bf T}_2)\in {{\bf P}_{c,{\bf n}}^-}(\cH)$ and  $k\in \NN$.

\bigskip

\section{ Sharper And\^ o type   inequalities  on varieties, universal models}

In this section, we provide sharper And\^ o type inequalities  for commuting contractions when at least one of them is of class $\cC_0$. In particular, for  commuting contractive matrices with spectrum in the open unit disc, we obtain an inequality which is sharper then Agler-McCarthy's inequality and Das-Sarkar's extension.  We   obtain  more general inequalities  for arbitrary commuting contractive matrices. Finally, we provide a model for  the universal operator algebra for two commuting contractions.

  In the particular case when $n=1$,  the noncommutative analytic Toeplitz algebra $F_n^\infty$ can be identified with the classical Hardy algebra
$H^\infty$. Every  $w^*$-closed  ideal $J\neq \{0\}$ of $H^\infty$  has the form $J_\psi=\psi H^\infty$, where $\psi\in H^\infty$ is an inner function uniquely determined up to a constant of modulus 1 (see \cite{Ga}). In this case
$$\cN_{J_\psi}=H^2\ominus \psi H^2=\ker \psi(S)^*,
$$
where  $S$ is the unilateral shift on the Hardy space $H^2$.
 An operator   $X\in B(\cH)$ is in the variety $\cV_{J_\psi}(\cH)$ if and only if  $X$ is  a completely non-unitary contraction such that $\psi(X)=0$.  The universal model of $\cV_{J_\psi}(\cH)$   is the truncated shift $B:=P_{\cN_{J_\psi}}S|_{\cN_{J_\psi}}$. When $J=\{0\}$, we take $\psi=0$. In this case, we take   $\cN_{\{0\}}=H^2$ and define  $\cV_{\{0\}}(\cH)$ to be  the set of all pure contraction $T\in B(\cH)$, i.e. ${T^*}^n\to 0$ as $n\to \infty$.
 The universal model of $\cV_{\{0\}}(\cH)$ is the unilateral shift $S$.

Recall that $\cU_{\bf T}$ is the set of all  unitary extensions of the isometry given by relation \eqref{iso1}. According to Lemma \ref{strong-limit}, for each $U\in \cU_{\bf T}$,
 the strong operator topology limit
$\varphi_{U^*}({\bf R}):=\text{\rm SOT-}\lim_{r\to 1}\varphi_{U^*}(r{\bf R})
$
 exists and defines a contractive  multi-analytic  operator.

\begin{theorem} \label{A-M2}
 Let $T_1\in \cV_{J_{\psi_1}}$ and $T_2\in \cV_{J_{\psi_2}}$ be  commuting contractions with    $d_i:=\dim \cD_{T_i}<\infty$. Then, $\varphi_{U_i^*}(S)$ is an isometric analytic Toeplitz operator on $H^2(\DD)\otimes \CC^{d_i}$ for any unitary operators  $U_1\in \cU_{(T_1,T_2)}$ and  $U_2\in \cU_{(T_2,T_1)}$,  and
$$
\|[p_{rs}(T_1,T_2)]_{k}\|\leq  \min \left\{\|[p_{rs}(B_1\otimes I_{\CC^{d_1}},\varphi_{U_1^*}(B_1))]_{k}\|, \|[p_{rs}(\varphi_{U_2^*}(B_2), B_2\otimes I_{\CC^{d_2}})]_{k}\|\right\}
$$
for any $ [p_{rs}]_{k}\in M_k(\CC[z,w])$,
where $B_i:=P_{\cN_{J_{\psi_i}}}S|_{\cN_{J_{\psi_i}}}$ is the universal  model of the variety $\cV_{J_{\psi_i}}$.
\end{theorem}

\begin{proof} Applying  Theorem \ref{ando}, in the particular case when $n_1=n_2=1$ and $U_1\in \cU_{(T_1,T_2)}$, we find  an  analytic Toeplitz operator $\varphi_{U_1^*}(S)$ on $H^2(\DD)\otimes \CC^{d_1}$  such that,  for any $ [p_{rs}]_{k}\in M_k(\CC[z,w])$,
$$
\|[p_{rs}(T_1,T_2)]_{k}\|\leq   \|[p_{rs}(B_1\otimes I_{\CC^{d_1}},\varphi_{U_1^*}(B_1))]_{k}\|.
$$
We prove that  $\varphi_{U_1^*}({\bf R})$ is  an isometry for any $U_1\in \cU_{(T_1, T_2)}$.
Indeed, in this case we have
$$U_1^*=\left[ \begin{matrix} A^*&C^*\\B^*&D^* \end{matrix}\right]:\begin{matrix}\cD_{T_1}\\\oplus\\
\cD_{T_2}\end{matrix}\to
\begin{matrix}\cD_{T_1}\\\oplus\\ \cD_{T_2}\end{matrix}
$$
Note that the condition \eqref{cond-iso} becomes
$$
\lim_{r\to 1}(1-r^2) \left\|\left(I-rS\otimes D^*\right)^{-1}(I\otimes B^*)x\right\|^2=0
$$
for any $x\in H^2(\DD)\otimes \cD_{T_1}$. We shall  check that this condition is satisfied. Since $D:\cD_{T_2}\to \cD_{T_2}$ is a contraction, there is an orthogonal decomposition $
\cD_{T_2}=\cH_u\oplus \cH_s$, such that $\cH_u$ and $\cH_s$ are reducing subspaces for $D$ with the property that  $D_1:=D|_{\cH_u}:\cH_u\to \cH_u$ is a unitary operator and $D_2:=D|_{\cH_s}:\cH_s\to \cH_s$ is a completely non-unitary contraction. Note that since $D_2$ is acting on a finite dimensional space, its spectral radius is strictly less than $1$. Indeed, assume that there is $\lambda\in \TT$ and $v\neq 0$ in $\cH_s$ such that $D_2v=\lambda v$. Since $D_2$ is a contraction, we also have $D_2^*v=\bar \lambda v$, which shows that the restriction of $D_2$ to the  span of $v$ is a unitary operator, contradicting that $D_2$ is completely non-unitary.
On the other hand, since $U_1$ is a unitary operator, we have $CC^*+DD^*=I$ and $AC^*+BD^*=0$. If $x\in \cH_u$, then $\|Dx\|=\|x\|$ and, due to the first relation above, we must have $C^*x=0$. Hence, and using the second relation, we deduce that $B(D^*x)=0$ for any $x\in \cH_u$, which shows that $\cH_u\subset \ker B$. Now, since $\cD_{T_2}=\text{\rm range} B^*\oplus \ker B$, we have
$$
 \text{\rm range}\, B^*=\ker B^{\perp}\subset \cH_u^\perp=\cH_s.
 $$
Consequently, for any $x\in H^2(\DD)\otimes \cD_{T_1}$,  we have
$y:=(I\otimes B^*)x\in H^2(\DD)\otimes \cH_s$ and, therefore,
$$
\left(I-rS\otimes D^*\right)^{-1}(I\otimes B^*)x=\left(I-rS\otimes D_2^*\right)^{-1}y.
$$
Since $\|S\otimes D_2\|<1$,  we have
$$\lim_{r\to 1} \left(I-rS\otimes D_2^*\right)^{-1}y=\left(I-S\otimes D_2^*\right)^{-1}y,
$$
and, consequently, relation \eqref{cond-iso} is satisfied, which proves that $\varphi_{U_1^*}({\bf R})$ is  an isometry for any $U_1\in \cU_{\bf T}$.

In a similar fashion, one can show that if $U_2\in \cU_{(T_2, T_1)}$ then
$\varphi_{U_2^*}(S)$ is an isometric analytic Toeplitz operator on $H^2(\DD)\otimes \CC^{d_2}$ and
$$
\|[p_{rs}(T_1,T_2)]_{k}\|\leq   \|[p_{rs}(\varphi_{U_2^*}(B_2), B_2\otimes I_{\CC^{d_2}})]_{k}\|
$$
for any $ [p_{rs}]_{k}\in M_k(\CC[z,w])$. The proof is complete.
\end{proof}

We recall \cite{SzFBK-book} that a contraction $T\in B(\cH)$ is called of class $\cC_0$ if $T$ is c.n.u. and there exists an inner function $u\in H^\infty$ such that $u(T)=0$. It is well-known that there is an essentially unique  minimal inner function $m_T\in H^\infty$ such that $m_T(T)=0$. It is clear that $T\in \cV_{J_{m_T}}(\cH)$ and the associated model is
$$
B=P_{\cN_{J_{m_T}}} S|_{\cN_{J_{m_T}}}, \qquad \cN_{J_{m_T}}=H^2\ominus m_TH^2.
$$

\begin{remark}
\label{rem}
Applying Theorem \ref{A-M2} when $T_1\in \cV_{J_{m_{T_1}}}(\cH)$ or $T_2\in \cV_{J_{m_{T_2}}}(\cH)$ are commuting $\cC_0$-contractions with
$d_i:=\dim \cD_{T_i}<\infty$, we obtain an improvement of And\^ o's  inequality \cite{An}. This is due to the fact that
$$
   \|[p_{rs}(B_1\otimes I_{\CC^{d_1}},\varphi_{U_1^*}(B_1))]_{k}\|
   \leq \|[p_{rs}(S\otimes I_{\CC^{d_1}},\varphi_{U_1^*}(S))]_{k}\|
   \leq \|[p_{rs}(U_1 ,U_2)]_{k}\|,
$$
where $B_1$ is the universal model of the variety $\cV_{J_{m_{T_1}}}$, and  $U_1$  and $U_2$ are commuting unitaries extending the isometries $S\otimes I_{\CC^{d_1}}$ and $\varphi_{U_1^*}(S)$, respectively.

  Moreover, the inequality in Theorem \ref{A-M2}  remains true if we drop the condition $d_i=\dim \cD_{T_i}<\infty$,  with the modification  that   $\varphi_{U_i^*}(S)$ is a contractive analytic Toeplitz operator on $H^2(\DD)\otimes \CC^{d_1}$. In particular (case $\psi_1=\psi_2=0$), if $T_1$ and $T_2$ are  commuting contractions such that $T_1$ is pure and $U\in \cU_{(T_1,T_2)}$, then $\varphi_{U^*}(S)$ is a contractive analytic operator on $H^2(\DD)\otimes \CC^{d_1}$  and
  $$
  \|[p_{rs}(T_1,T_2)]_{k}\|\leq \|[p_{rs}(S\otimes I_{\CC^{d_1}},\varphi_{U_1^*}(S))]_{k}\|
   \leq \|[p_{rs}(U_1 ,U_2)]_{k}\|,
  $$
  where   $U_1$ and $U_2$ are  commuting unitary dilations for  $S\otimes I_{\CC^{d_1}}$ and $\varphi_{U_1^*}(S)$, respectively.
\end{remark}

Let $T$   be  a  contractive matrix and assume that it has  no eigenvalues of modulus 1. This is equivalent to saying that the contractive matrix is completely non-unitary (c.n.c). Let  $m_{T}(z)=(z-\lambda_1)^{n_1}\cdots (z-\lambda_k)^{n_k}$ be the minimal polynomial of $T$ and  let $J_{T}$ be  the $w^*$-closed two sided ideal of $H^\infty$ generated by $m_{T}$.  Then $J_{T}=b_{T} H^\infty$,
where $b_{T}$ is the Blaschke product defined by
$$
b_{T}(z):=\prod_{i=1}^k\left(\frac{z-\lambda_i}{1-\bar\lambda_i z}\right)^{n_i}.
$$
In this case, we have $\cN_{J_{T}}=\cN_T:=H^2\ominus b_{T}H^2$ and it is well known that   
$$\cN_{T}=\text{\rm span}\, \{\ker (S^*-\bar\lambda_i)^{n_i}: i=1,\ldots k\}\quad \text{ and } \quad \dim\cN_{T}=n_1+\cdots +n_k
$$
(see e.g. \cite{N}).
Moreover, $\ker (S^*-\bar\lambda_i)^{n_i}=\text{\rm span} \{v_j^i\}_{j=0}^{n_i-1}$, where
$v^i_j(z):=\sum_{m=j}^\infty \frac{m !}{(m-j)!} \bar\lambda_i^{m-j}z^m$.
We call the truncated shift $B:=P_{\cN_T} S|_{\cN_T}$ the universal model associated with the c.n.c. contractive matrix $T$.

\begin{theorem}\label{AM3} Let $T_1$ and $T_2$ be  commuting contractive matrices such that $T_1$ is completely non-unitary  and let  $d_1:=\dim (I-T_1T_1^*)^{1/2}\cH$.
For any unitary operator  $U\in \cU_{(T_1, T_2)}$, $\varphi_{U^*}(S)$ is an isometric analytic Toeplitz operator on $H^2(\DD)\otimes \CC^{d_1}$   and
$$
\|[p_{rs}(T_1,T_2)]_{k}\|\leq  \|[p_{rs}(B_1\otimes I_{\CC^{d_1}},\varphi_{U^*}(B_1))]_{k}\|, \qquad [p_{rs}]_{k}\in M_k(\CC[z,w]),
$$
where $B_1:=P_{\cN_{{T_1}}}S|_{\cN_{{T_1}}}$ is the model operator associated with $T_1$.
\end{theorem}
\begin{proof} Since $T_1$ is a  completely non-unitary  contraction on a finite-dimensional Hilbert space,  we can use the Jordan canonical form  to show that $T_1$  is   a pure contraction, i.e. $(T_1^*)^n\to 0$ stronly, as $n\to \infty$. Applying Theorem \ref{A-M2} to  the commuting contractive matrices $T_1$ and $T_2$, when $\psi_1=m_{{T}_1}$, the minimal polynomial of ${T}_1$, and $\psi_2=0$, the result follows.
\end{proof}

\begin{corollary} \label{2-pure}
Let $T_1$ and $T_2$ be  commuting contractive matrices with spectrum in the open unit disk $\DD$ with $d_i:=\dim \cD_{T_i}$. Then, for any polynomial $p$ in two variables,
$$
\|p(T_1, T_2)\|\leq \inf_{U, W}\left\{ \|p(B_1\otimes I_{\CC^{d_1}}, \varphi_{U^*}(B_1))\|, \|p(\varphi_{W^*}(B_2), B_2\otimes I_{\CC^{d_2}})\|\right\},
$$
where $B_i$ is the universal model for $T_i$ and the infimum is taken over all $U\in \cU_{(T_1,T_2)}$ and $W\in \cU_{(T_2, T_1)}$.
\end{corollary}
\begin{proof} Applying Theorem \ref{AM3} to  the pairs  $(T_1, T_2)$ and $(T_2, T_1)$ of commuting contractive matrices, the result follows.
\end{proof}

A few remarks are necessary. If $T_1$ and $T_2$  are commuting contractive matrices with no eigenvalues of modulus 1,  Agler-McCarthy  proved, in their remarkable paper  \cite{AM}, that the pair $(T_1, T_2)$ has a  co-isometric  extension $(M_z^*, M_\Phi^*)$   on $H^2\otimes \CC^d$  and, for any polynomial $p$ in two variables,
\begin{equation}\label{A-M}
\|p(T_1, T_2)\|\leq \|p(M_z\otimes I_{\CC^{d}}, M_\Phi)\|\leq \|p\|_V,
\end{equation}
where $V$ is a distinguished variety in $\DD^2$ depending on $T_1$ and $T_2$,  defined by
$$V:=\{(z,w)\in \DD^2: \ \det(\Phi(z)-wI_{\CC^d})\}.
$$
We remark that
 the inequalities provided in Theorem \ref{AM3} and Corollary \ref{2-pure} are   sharper than And\^ o's inequality and Agler-McCarthy's  inequality,    due to the fact that
 $$
\|P_{\cN_{T_1}\otimes \CC^d}p(M_z\otimes I_{\CC^{d}}, M_\Phi)|_{\cN_{T_1}\otimes \CC^d}\|\leq \|p(M_z\otimes I_{\CC^{d}}, M_\Phi)\|\leq \|p(U_1,U_2)\|,
$$
where $U_1$ and $U_2$ are commuting unitaries extending the isometries
$M_z\otimes I_{\CC^{d}}$ and $M_\Phi$, respectively.
We shall consider  some examples.

\begin{example} Consider the polynomial
$$
p(z, w)= m_{T_1} w^m+m_{T_1}^2  w^k, \qquad m\leq k.
$$
 Since  $B:=P_{\cN_{{T_1}}}S|_{\cN_{{T_1}}}$ and   $m_{T_1}(B)=0$,  Theorem \ref{AM3} implies
$\|p(B\otimes I_{\CC^{d}},\varphi_{U^*}(B)))\|=0$.
On the other hand, $\|p(M_z\otimes I_{\CC^{d}}, M_\Phi)\|>0$. Indeed, if we assume the contrary, then
$$(M_{b_{T_1}}\otimes I)M_\Phi^m[I+ (M_{b_{T_1}}\otimes I)M_\Phi^{k-m}]=0
$$
Since $M_{b_{T_1}}$ and $M_\Phi$ are isometries, we must have
$(M_{b_{T_1}}^*\otimes I){M_\Phi^*}^{k-m}=-I$. Since $M_{b_{T_1}}^*$ is not injective, we obtain a contradiction.
\end{example}

\begin{example}
Let $T_1$ be a nilpotent contractive matrix such that $T_1^k=0$ and $T_1^{k-1}\neq 0$, where $k\geq 2$. Given $a, b>0$ and $\ell\in \NN$, consider the polynomial  in two variables
$$
p(z,w)=az+bz^k w^\ell.
$$
If $\varphi$ is any inner analytic  function on $H^2\otimes \CC^d$, then  $(M_z^{k-1}\otimes I)\varphi(M_z)^{\ell-1}$ is a pure isometry and, consequently,
$$\|p(M_z\otimes I_{\CC^{d}}, M_\Phi)\|=\|aI+b(M_z^{k-1}\otimes I)\varphi(M_z)^{\ell-1}||=\sup_{z\in \DD} |a+bz|=a+b.
$$
Agler-McCarthy inequality \eqref{A-M} becomes
$$
\|p(T_1, T_2)\|\leq a+b\leq \|p\|_V.
$$
On the other hand, since $T_1$ is in  the variety
$\cV_k:=\{X\in B(\cH): X^k=0\}$  and the model operator  for $\cV_k$  is  $B:=P_{\cP_{k-1}} S|_{\cP_{k-1}}$, where $S$ is the unilateral shift and $\cP_{k-1}:=\text{\rm span }\{1, z,\ldots, z^{k-1}\}\subset H^2$, Theorem \ref{AM3} implies
$$
\|p(T_1, T_2)\|\leq \|p(B\otimes I_{\CC^d}, \varphi(B))\|=\|aB\|=a.
$$
\end{example}

\begin{example} Let $\cC^k$, $k\geq 0$, be the set of all nilpotent contractive matrices of order $k$, and consider the polynomial
$$
p(z,w)=w+zw+z^2w+\cdots+z^nw.
$$
If $2\leq k\leq n$ and $\varphi(M_z)$ is any inner analytic operator on $H^2\otimes \CC^d$, then
$$
\|p(M_z\otimes I_{\CC^{d}}, \varphi(M_z))\|=\|I+M_z+\cdots + M_z^n\|=\sup_{z\in \DD}
|1+z+\cdots +z^n|=n+1.
$$
Therefore, if $T_1\in \cC^k$, then Agler-McCarthy inequality becomes
$$
\|p(T_1, T_2)\|\leq n+1\leq \|p\|_V.
$$
On the other hand, we have
\begin{equation*}
\begin{split}
\|p(B\otimes I_{\CC^d}, \varphi(B))\|&=\|(I+B+\cdots+ B^{k-1})\varphi(B)\|\\
&\leq \|I+B+\cdots+ B^{k-1}\|\leq k,
\end{split}
\end{equation*}
where $B:=P_{\cP_{k-1}} S|_{\cP_{k-1}}$ is the model for the varierty $\cC^k$.
 Theorem \ref{AM3} implies
$$
\|p(T_1, T_2)\|\leq \|p(B\otimes I_{\CC^d}, \varphi(B))\|\leq k.
$$
\end{example}

\begin{proposition} \label{unitary-pure}
Let $T_1$ and $T_2$ be  commuting contractive matrices  such that $\sigma(T_1)\subset \TT$ and $T_2$ is completely non-unitary. Then, for any polynomial $p$ in two variables,
$$
\|p(T_1, T_2)\|\leq  \max_{\lambda\in \sigma(T_1)}\|p(\lambda, B_{2,\lambda} )\|\leq   \max_{\lambda\in \sigma(T_1)}\|p(\lambda, B_2)\|
$$
where   $B_{2,\lambda}$ is the universal model for $T_2|_{\cE_\lambda}$ and
$\cE_\lambda$ is the eigenspace corresponding to the unimodular eigenvalue $\lambda\in \sigma(T_1)$.
\end{proposition}
\begin{proof}
Assume that  $T_1$ and $ T_2$ are acting on a finite dimensional Hilbert space $\cH$. Since $\sigma(T_1)\subset \TT$, one can easily see that $T_1$ is a unitary operator, Indeed, this is due to the well-known fact that a contraction $T_1$ and its adjoint $T_1^*$ have the same invariant vectors, i.e. $T_1h=h$ if and only if $T_1^*h=h$ (see \cite{SzFBK-book}). For each $\lambda\in \sigma(T_1)$, let $\cE_\lambda$ be the corresponding eigenspace. Since $T_2$ commutes  $T_1$, which is   unitary operator, it commutes with each spectral projection $P_{\cE_\lambda}$. Therefore, each eigenspace $\cE_\lambda$ is reducing for $T_2$ as well. With respect to the decomposition $\cH=\bigoplus_{\lambda\in \sigma(T_1)} \cE_\lambda$ we have
$T_1=\bigoplus_{\lambda\in \sigma(T_1)} \lambda P_{\cE_\lambda}$  and
$T_2=\bigoplus_{\lambda\in \sigma(T_1)} T_2|_{\cE_\lambda}$.
For any polynomial $p$ in two variables,
we have
$$p(T_1,T_2)=\bigoplus_{\lambda\in \sigma(T_1)} p(\lambda I_{\cE_\lambda}, T_2|_{\cE_\lambda})\quad \text{and} \quad
\|p(T_1,T_2)\|=\max_{\lambda\in \sigma(T_1)}\| p(\lambda I_{\cE_\lambda}, T_2|_{\cE_\lambda})\|.
$$
Since $T_2$ is c.n.u., so is each operator $T_2|_{\cE_\lambda}$. If $B_{2,\lambda}$ is the universal model for $T_2|_{\cE_\lambda}$, then, applying Theorem \ref{AM3} to the pair $(T_2|_{\cE_\lambda})$, we deduce that $\| p(\lambda I_{\cE_\lambda}, T_2|_{\cE_\lambda})\|\leq \|p(\lambda, B_{2,\lambda} )\|$, which proves the first inequality of the proposition.
To prove the second one, note that the minimal polynomial of $T_2|_{\cE_\lambda}$ divides the minimal polynomial of $T_2$ and, consequently, $B_{2,\lambda}=P_{\cN} B_2 |_{\cN}$, where $\cN:=\cN_{m_{T_2|_{\cE_\lambda}}}$. This completes the proof.
\end{proof}

\begin{proposition} \label{2-unitary}
Let $T_1$ and $T_2$ be  commuting contractive matrices with spectrum in the unit circle $\TT$. Then
$$
\|p(T_1, T_2)\| =\max_{(\lambda, \mu)\in \Omega_{(T_1,T_2)}} |p(\lambda, \mu)|,
$$
where  $$\Omega_{(T_1,T_2)}:=\left\{ (\lambda, \mu)\in \sigma(T_1)\times \sigma(T_2): \ \cE_\lambda^1\cap \cE_\mu^2\neq \{0\}\right\},
$$ and
$\cE_\lambda^1$ and $ \cE_\mu^2$ are the eigenspaces corresponding to the eigenvalues $\lambda\in \sigma(T_1)$ and $\mu\in \sigma(T_2)$, respectively.
\end{proposition}
\begin{proof}
As in the proof of Proposition \ref{unitary-pure}, one can see that $T_1$ and $T_2$ are commuting unitary matrices. Using the spectral theorem, we obtain
$$
p(T_1, T_2)=\sum_{(\lambda, \mu)\in \Omega_{(T_1,T_2)}} p(\lambda, \mu)P_{\cE_\lambda^1\cap \cE_\mu^2}.
$$
Since $\{P_{\cE_\lambda^1\cap \cE_\mu^2}\}_{(\lambda, \mu)\in \Omega_{(T_1,T_2)}}$ is a set of orthogonal projections, the result follows.
\end{proof}

We denote by $\cC_u$  (resp. $\cC_{cnc}$) the class of all  unitary (resp. completely non-unitary) operators on a Hilbert space.
\begin{lemma} \label{structure}
Let $T_1$ and $T_2$ be commuting contractive operators on a finite dimensional Hilbert space $\cH$. Then there is an orthogonal  decomposition
$\cH=\oplus_{i=1}^4 \cH_i$ such that each subspace $\cH_i$ is reducing for both $T_1$ and $T_2$ and, setting
$$T_j^{(i)}:=T_j|_{\cH_i}:\cH_i\to \cH_i, \qquad j=1,2,
$$
we have $T_j=\oplus_{i=1}^4 T_j^{(i)}$  with the properties that \ $T_1^{(i)} T_2^{(i)}=T_2^{(i)} T_1^{(i)}$ for $i\in \{1,\ldots, 4\}$,
  $$
  T_1^{(1)}, T_1^{(2)}\in \cC_u; \quad   T_1^{(3)}, T_1^{(4)}\in \cC_{cnu}
  $$
  and
  $$ T_2^{(1)},  T_2^{(3)}\in \cC_{u}; \quad T_2^{(2)} , T_2^{(4)}\in \cC_{cnu}.
  $$
 \end{lemma}
 \begin{proof}  First, we decompose $T_1$ as a direct sum of its unitary part $A_1:=T_1|_{\cH_u}$  and the c.n.u. part $B_1:=T_1|_{\cH_{cnu}}$. Therefore,  $\cH=\cH_{u}\oplus \cH_{cnu}$,
 $$
 T_1=\left[\begin{matrix} A_1&0\\0&B_1\end{matrix}\right],
 $$
 and  the decomposition is unique \cite{SzFBK-book}.
  Since the spectrum of $A_1$  consists of eigenvalues of modulus one  and $T_1$ commutes with $T_2$, we can deduce that $\cH_u$ is a reducing subspace for $T_2$. Consequently, we have
 $$
 T_2=\left[\begin{matrix} A_2&0\\0&B_2\end{matrix}\right],
 $$
 where $A_2:=T_2|_{\cH_u}$ and  $B_2:=T_2|_{\cH_{cnu}}$. Note that
 $A_1A_2=A_2A_1$ and $B_1B_2=B_2B_1$. The next step is to decompose the contractive matrix  $A_2$ as a direct sum of its unitary part $T_2^{(1)}:=A_2|_{\cH_1}\in \cC_u$ and c.n.u. part $T_2^{(2)}:=A_2|_{\cH_{cnu}}\in \cC_{cnu}$. Therefore, $\cH_u=\cH_1\oplus \cH_2$ and
 $$
 A_2=\left[\begin{matrix} T_2^{(1)}&0\\0&T_2^{(2)}\end{matrix}\right].
 $$
 Since $A_1A_2=A_2A_1$ and the spectrum of $A_2$  consists of eigenvalues of modulus one, as above, we  deduce that $\cH_1$  and $\cH_2$ are reducing subspaces for $A_1$ and, therefore,
 $$
 A_1=\left[\begin{matrix} T_1^{(1)}&0\\0&T_1^{(2)}\end{matrix}\right],
 $$
 where $T_1^{(1)}:=A_1|_{\cH_1}\in \cC_u$ and  $T_1^{(2)}:=A_1|_{\cH_2}\in \cC_u$.

 Since $B_2:=T_2|_{\cH_{cnu}}$ is a contraction on ${\cH_{cnu}}$, we decompose it as well  as a direct sum of its unitary part $T_2^{(3)}:=B_2|_{\cH_3}\in \cC_u$ and c.n.u. part $T_2^{(4)}:=B_2|_{\cH_4}\in \cC_{cnu}$.
  Therefore, $\cH_{cnu}=\cH_3\oplus \cH_4$ and
 $$
 B_2=\left[\begin{matrix} T_2^{(3)}&0\\0&T_2^{(4)}\end{matrix}\right].
 $$
 Once again,  since  the spectrum of $T_2^{(3)}$  consists of eigenvalues of modulus one and  $B_1B_2=B_2B_1$, we deduce that
 $\cH_3$  and $\cH_4$ are reducing subspaces for $B_1$ and, therefore,
 $$
 B_1=\left[\begin{matrix} T_1^{(3)}&0\\0&T_1^{(4)}\end{matrix}\right],
 $$
 where $T^{(3)}_1:=B_1|_{\cH_3}\in \cC_{cnu}$ and  $T^{(4)}_1:=B_1|_{\cH_4}\in \cC_{cnu}$. Putting together the decompositions above, we obtain that
 $T_j=\oplus_{i=1}^4 T_j^{(i)}$, $j=1,2$, where  each $T_j^{(i)}$ has the required property. The proof is complete.
 \end{proof}

Using   Lemma \ref{structure}, Proposition \ref{2-unitary}, Proposition \ref{unitary-pure}, and Corollary \ref{2-pure},   we obtain the following inequality for commuting contractive matrices.

\begin{theorem} \label{general} Let $T_1$ and $T_2$ be   commuting contractive matrices and let $T_j=\oplus_{i=1}^4 T_j^{(i)}$, $j=1,2$, be the decomposition from  Lemma \ref{structure}. Then,
 for any polynomial $p$ in two variables,
$$
\|p(T_1, T_2)\|\leq  \max_{i\in \{1,2,3,4\}}\|p(T_1^{(i)}, T_2^{(i)})\|,
$$
where
\begin{equation*}
\begin{split}
\|p(T_1^{(1)}, T_2^{(1)})\|&\leq \max_{(\lambda, \mu)\in \Omega_{(T_1^{(1)},T_2^{(1)})}} |p(\lambda, \mu)|,\\
\|p(T_1^{(2)}, T_2^{(2)})\|&\leq \max_{\lambda\in \sigma(T_1^{(2)})}\|p(\lambda, B_{T_2^{(2)},\lambda} )\|,\\
\|p(T_1^{(3)}, T_2^{(3)})\|&\leq \max_{\mu\in \sigma(T_2^{(3)})}\|p(  B_{T_1^{(3)},\mu}, \mu )\|,\\
\|p(T_1^{(4)}, T_2^{(4)})\|&\leq
 \inf_{U, W}\left\{ \|p(B_{T_1^{(4)}}, \varphi_{U^*}(B_{T_1^{(4)}}))\|, \|p(\varphi_{W^*}(B_{T_2^{(4)}}), B_{T_2^{(4)}})\|\right\},
\end{split}
\end{equation*}
where the set $\Omega_{(T_1^{(1)},T_2^{(1)})}$ is defined as in Proposition \ref{2-unitary},  $B_{T_j^{(i)}}$ is the universal model for $T_j^{(i)}$, the operators  $ B_{T_2^{(2)},\lambda}$ and $B_{T_1^{(3)},\mu}$ are defined  as in Proposition \ref{unitary-pure},  and the infimum is taken over all unitary operators $U\in \cU_{(T_1^{(4)},T_2^{(4)})}$ and $W\in \cU_{(T_2^{(4)}, T_1^{(4)})}$.
\end{theorem}

A weaker version of Theorem \ref{general} is the following.

\begin{corollary} If $T_1$ and $T_2$ are  commuting contractive matrices, then there exist  $d_1, d_2\in \NN$ and   isometric analytic Toeplitz operators $\varphi_i(S)$ on $H^2(\DD)\otimes \CC^{d_i}$ such that, for any polynomial $p$ in two variables,
$$
\|p(T_1, T_2)\|\leq \max_{\lambda_i\in \sigma(T_i)\cap \TT}\left\{ |p(\lambda_1, \lambda_2)|, \|p(\lambda_1, B_2)\|,  \|p(  B_1, \lambda_2)\|, \min\{\|p(B_1, \varphi_1(B_1))\|,\|p(B_2, \varphi_2(B_2))\|\} \right\},
$$
where $B_i$ is the universal model for the completely non-unitary  part of $T_i$, under the convention that if $\sigma(T_1)\cap \TT=\emptyset$ or
$\sigma(T_2)\cap \TT=\emptyset$, then the corresponding terms in the right-hand side of  inequality are deleted.
\end{corollary}
\begin{proof}
First, note that $\sigma(T_1^{(1)}\oplus T_1^{(2)})=\sigma(T_1)\cap \TT$ and $\sigma(T_2^{(1)}\oplus T_2^{(3)})= \sigma(T_2)\cap \TT$. Hence,
$$\Omega_{(T_1^{(1)},T_2^{(1)})}\subset [\sigma(T_1)\cap \TT]\times [\sigma(T_2)\cap \TT].
$$
On the other hand, note that the minimal polynomial of $T_2^{(2)}|_{\cE_\lambda}$ divides the minimal polynomial of $T_2^{(2)}$, which divides the minimal polynomial of $T_2^{(2)}\oplus T_2^{(4)}$.
 Consequently,  the universal model $B_{T_2^{(2)},\lambda}$ has the property that $B_{T_2^{(2)},\lambda}^*$ is the restriction of $B_2^*$ to the appropriate invariant subspace. Hence, we obtain
 $$ \max_{\lambda\in \sigma(T_1^{(2)})}\|p(\lambda, B_{T_2^{(2)},\lambda} )\|\leq   \max_{\lambda\in \sigma(T_1^{(2)})}\|p(\lambda, B_{2})\|.
 $$
Similarly, we can prove that
$$
\|p(  B_{T_1^{(3)},\mu}, \mu )\|\leq   \max_{\lambda\in \sigma(T_1^{(2)})}\|p(\lambda, B_{2})\|
$$
and
$$
\min\left\{ \|p(B_{T_1^{(4)}}, \varphi_{U^*}(B_{T_1^{(4)}}))\|, \|p(\varphi_{W^*}(B_{T_2^{(4)}}), B_{T_2^{(4)}})\|\right\}\leq
\min\{\|p(B_1, \varphi_1(B_1))\|,\|p(B_2, \varphi_2(B_2))\|\}
$$
Now, we can use Theorem
\ref{general} to complete the proof.
\end{proof}

Denote by ${\bf D}^2(\cH)$  the set of all pairs $(T_1, T_2)$ of commuting contractions on a Hilbert space $\cH$ and
the the abstract bidisk
$${\bf D}^2:=\{{\bf D}^2(\cH): \ \cH \text{ is a Hilbert space }\}.
$$
For simplicity,   we write ${\bf T}\in {\bf D}^2$ if ${\bf T}$ is a pair $(T_1, T_2)$ of commuting contractions on a Hilbert space. We denote by ${\bf D}_{F}^2 $  the set of all pairs $(T_1, T_2)$ of commuting contractions acting on a finite dimensional  Hilbert space.
Given $k\in \NN$ and any $k\times k$  matrix $[p_{ij}]_k$ of polynomials in $\CC[z,w]$, we define
$$
\|[p_{ij}]_k\|_u:=\sup_{{\bf T}\in {\bf D}^2}\|[p_{ij}({\bf T})]_k\|\quad \text{ and }\quad \|[p_{ij}]_k\|_F:=\sup_{{\bf T}\in {\bf D}_F^2}\|[p_{ij}({\bf T})]_k\|.
$$

The following lemma is a matrix extension of a result from \cite{Dr2}.
\begin{lemma}\label{finite} For any $k\times k$  matrix $[p_{ij}]_k$ of polynomials in $\CC[z,w]$,
$$ \|[p_{ij}]_k\|_u=\|[p_{ij}]_k\|_F.
$$
\end{lemma}
\begin{proof}
Let $(T_1, T_2)$ be a pair of commuting contractions on a Hilbert space $\cH$. For each $i,j\in \{1,\ldots, k\}$, we  fix a polynomial
$$
p_{ij}(z,w)=\sum_{n_1,n_2\in \ZZ_+, n_1+n_2\leq N} a^{ij}_{n_1, n_2} z^{n_1} w^{n_2}
$$
in $\CC[z,w]$, and assume that $\|[p_{ij}]_k\|_F\leq 1$.
Let $\cM\subset H^2(\TT)$ be the subspace of all analytic trigonometric polynomials
of the form
$g(e^{it})=\sum_{k\in \ZZ_+, |k|\leq m} c_k e^{ikt}$
and let $S_\cM:=P_\cM M_{e^{it}}|_\cM$ be the compression of the unilateral shift to the subspace $\cM$. Note that the operators $N_1:=S_\cM\otimes T_1$ and $N_2:=S_\cM\otimes T_2$ are commuting contractions acting on $\cM\otimes \cH$.
Let $f(e^{it}):=\frac{1}{\sqrt{m+1}} \sum_{k\in \ZZ_+, |k|\leq m} e^{ikt}$ and let $\xi_1,\ldots, \xi_k\in \cH$ be such that $\sum_{i=1}^k |\xi_i|^2\leq 1$.
Define the finite dimensional subspace
$$
\cN:=\text{\rm span} \left\{ N_1^{n_1} N_2^{n_2} (f\otimes \xi_j): \ n_1, n_2\in \ZZ_+, j\in \{1,\ldots, k\}\right\}.
$$
Since $N_1|_\cN$ and $N_2|_\cN$ are communing contractions acting on the finite dimensional space $\cN$ and $\|[p_{ij}]_k\|_F\leq 1$, we have
\begin{equation} \label{inner}
\left< [p_{ij}(N_1, N_2)]_k\left[\begin{matrix} f\otimes \xi_1\\ \vdots \\f\otimes \xi_k\end{matrix}\right], \left[\begin{matrix} f\otimes \eta_1\\ \vdots \\f\otimes \eta_k\end{matrix}\right]\right>\leq 1
\end{equation}
for any  $\eta_1,\ldots, \eta_k\in \cH$  such that   $\sum_{i=1}^k |\eta_i|^2\leq 1$.
Setting
$$
K_m(p_{ij})(z,w):=\sum_{n_1,n_2\in \ZZ_+, n_1+n_2\leq N} \left(1-\frac{n_1+n_2}{m+1}\right)a^{ij}_{n_1, n_2} z^{n_1} w^{n_2},
$$
straightforward calculations reveal that
$$
\left<p_{ij}(N_1, N_2)(f\otimes \xi_j), f\otimes \xi_i\right>=
\left<K_m(p_{ij})(T_1,T_2)\xi_j,\eta_i\right>
$$
if $m>N$.
Consequently,  relation \eqref{inner} implies
$$
\left< [K_m(p_{ij})(T_1, T_2)]_k\left[\begin{matrix} \xi_1\\ \vdots \\\xi_k\end{matrix}\right], \left[\begin{matrix} \eta_1\\ \vdots \\ \eta_k\end{matrix}\right]\right>\leq 1,
$$
for any  $\xi_1,\ldots, \xi_k\in \cH$  and  $\eta_1,\ldots, \eta_k\in \cH$  such that $\sum_{i=1}^k |\xi_i|^2\leq 1$ and $\sum_{i=1}^k |\eta_i|^2\leq 1$, respectively. Hence, we deduce that
$\|[K_m(p_{ij})]_k\|_u\leq 1$.

On the other hand, we have
$$
\|K_m(p_{ij})(T_1, T_2)- p_{ij}(T_1,T_2)\|\leq
\sum_{n_1,n_2\in \ZZ_+, n_1+n_2\leq N} \frac{n_1+n_2}{m+1}|a^{ij}_{n_1, n_2}|
$$
 for any $(T_1, T_2)\in {\bf D}^2(\cH)$.
This shows that
$\|K_m(p_{ij}) - p_{ij} \|_u\to 0$ as $m\to\infty$ and, consequently,
$$
\|[K_m(p_{ij}) - p_{ij}]_k \|_u\leq \left(\sum_{i,j} \|K_m(p_{ij}) - p_{ij} \|_u^2\right)^{1/2}\to 0,\qquad \text{ as } m\to\infty.
$$
Since,
$$
\|[p_{ij}]_k\|_u\leq \|[p_{ij}]_k-[K_m(p_{ij})]_k\|_u +\|[K_m(p_{ij})]_k\|_u
$$
and $\|[K_m(p_{ij})]_k\|_u\leq 1$ for any $m\to \infty$, we conclude that
$\|[p_{ij}]_k\|_u\leq 1$. This shows that
$\|[p_{ij}]_k\|_u\leq \|[p_{ij}]_k\|_F$. Since the reverse inequality is clear, the proof is complete.
\end{proof}

A closer look at the proof of Lemma \ref{finite} reveals that
$$ \|[p_{ij}]_k\|_u=\|[p_{ij}]_k\|_N:=\sup_{{\bf T}\in {\bf D}_N^2}\|[p_{ij}({\bf T})]_k\|,
$$
where ${\bf D}_N^2$ stands for the set of all pairs $(T_1, T_2)$ of nilpotent commuting contractions. We also remark that Lemma \ref{finite} can be easily extended, with the same proof, to the abstract polydisk ${\bf D}^n$.

\begin{theorem} \label{ando-bidisk}  There   is an  isometric  analytic Toeplitz operator $\varphi(S)\in B(H^2(\DD)\otimes \ell^2)$   such that
$$
\|[p_{rs}({T}_1,{T}_2)]_{k}\|\leq  \|[p_{rs}(S\otimes I_{ \ell^2},\varphi(S))]_{k }\|,
$$
 for any commuting contractions ${T}_1, {T}_2\in B(\cH)$,  any   $k\times k$ matrix $[p_{rs}]_{k}$ of polynomials in $\CC[z,w]$, and  any $k\in \NN$.
\end{theorem}
\begin{proof}
According to Lemma \ref{finite},
$ \|[p_{ij}]_k\|_u=\|[p_{ij}]_k\|_F$ for any   $k\times k$  matrix  $[p_{ij}]_k$  of polynomials in $\CC[z,w]$. Hence, we deduce that
\begin{equation} \label{pure}
\|[p_{ij}]_k\|_u=\sup_{(T_1,T_2)\in {\bf D}_F^2, r\in [0,1)}\|[p_{ij}(rT_1,T_2)]_k\|=\sup_{(T_1,T_2)\in {\bf D}_F^2,  T_1, T_2 \text{ pure}}\|[p_{ij}(T_1,T_2)]_k\|.
\end{equation}
Fix $[p_{ij}]_k\in M_k(\CC[z,w])$.
Due to relation \eqref{pure}, we can choose a sequence $\left\{(T_1^{(m)}, T_2^{(m)})\right\}_{m=1}^\infty$  in ${\bf D}^2_F$ with $T_1^{(m)}$ and $T_2^{(m)}$ pure contractions such that
\begin{equation}
\label{sup}
\|[p_{ij}]_k\|_u=\sup_{m}\|[p_{ij}(T_1^{(m)}, T_2^{(m)})]_k\|
\end{equation}
Due to Theorem \ref{A-M2}, when $\cV_{J_\psi}=B(\cH)$, for each $m\in \NN$, we find $d(m)\in \NN$ and an isometric analytic Toeplitz operator $\varphi_m(S)$ on the Hardy space $H^2(\DD)\otimes \CC^{d(m)}$ such that
$$
\|[p_{ij}(T_1^{(m)},T_2^{(m)})]_k\|\leq  \|[p_{ij}(S\otimes I_{\CC^{d(m)}},\varphi_m(S))]_{k }\|.
$$
Therefore, relation \eqref{sup} implies
\begin{equation*}
\begin{split}
\|[p_{ij}]_k\|_u&=\|[p_{ij}(\oplus_{m=1}^\infty T_1^{(m)}, \oplus_{m=1}^\infty T_2^{(m)})]_k\|\\
&\leq
\|[p_{ij}(\oplus_{m=1}^\infty  (S\otimes I_{\CC^{d(m)}}), \oplus_{m=1}^\infty \varphi_m(S))]_k\|\leq \|[p_{ij}]_k\|_u.
\end{split}
\end{equation*}
Hence, we deduce that
\begin{equation}
\label{ma}
\|[p_{ij}]_k\|_u= \|[p_{ij}(S\otimes I_{ \ell^2},\psi(S))]_{k }\|,
\end{equation}
where $\psi(S):=\oplus_{m=1}^\infty \varphi_m(S))$ is
an  isometric analytic Toeplitz operator   on the Hardy space $H^2(\DD)\otimes \ell^2$.
Let $\CC_\QQ[z,w]$ be the set of all polynomials with coefficients in $\QQ+i\QQ$.
Since the set
$\{[p_{ij}]_k: \ p_{ij}\in \CC_\QQ[z,w]\}$ is countable, let
$
[p^{(1)}_{ij}]_k, [p^{(2)}_{ij}]_k, \ldots$ be an enumeration of it.
Due to relation \eqref{ma}, for each $s\in \NN$, there is an isometric  analytic Toeplitz operator  $\psi_s(S)$ on the Hardy space $H^2(\DD)\otimes \ell^2$ such that
\begin{equation}
\label{ma2}
\|[p^{(s)}_{ij}]_k\|_u= \|[p_{ij}^{(s)}(S\otimes I_{ \ell^2},\psi_s(S))]_{k }\|,\qquad s\in \NN.
\end{equation}
Define the isometric analytic operator $\Gamma_k(S):=\oplus_{s=1}^\infty \psi_s(S)$ acting on $H^2(\DD)\otimes \ell^2$ and let us prove that
\begin{equation}
\label{qij}
\|[q_{ij}]_k\|_u= \|[q_{ij}(S\otimes I_{ \ell^2},\Gamma_k(S))]_{k }\|
\end{equation}
for any $[q_{ij}]_k\in M_k(\CC[z,w])$.
First, note that relation \eqref{ma2} implies
\begin{equation}
\label{Ga}
\|[p^{(s)}_{ij}]_k\|_u= \|[p_{ij}^{(s)}(S\otimes I_{ \ell^2},\Gamma_k(S))]_{k }\|\qquad \text{for any } s\in \NN.
\end{equation}
 Fix $[q_{ij}]_k\in M_k(\CC[z,w])$ and $\epsilon >0$,  and choose $[p^{(s_0)}_{ij}]_k$ such that
 \begin{equation}\label{u}
 \left\|[q_{ij}]_k-[p^{(s_0)}_{ij}]_k\right\|_u<\epsilon.
 \end{equation}
 Due to relations \eqref{ma}, \eqref{u}, and \eqref{Ga}, we deduce that there exists an isometric analytic Toeplitz operator $\psi^{(q)}(S)$ on $H^2(\DD)\otimes \ell^2$ such that
 \begin{equation*}
 \begin{split}
 \|[q_{ij}]_k\|_u&=\|[q_{ij}(S\otimes I_{ \ell^2},\psi^{(q)}(S))]_{k }\|\\
 &
\leq\|[p_{ij}^{(s_0)}(S\otimes I_{ \ell^2},\psi^{(q)}(S))]_{k }\| +\epsilon\\
 &\leq \|[p_{ij}^{(s_0)}]_k\|_u+\epsilon\\
 &=
 \|[p_{ij}^{(s_0)}(S\otimes I_{ \ell^2},\Gamma_k(S))]_{k }\|+\epsilon\\
 &\leq
 \|[q_{ij}(S\otimes I_{ \ell^2},\Gamma_k(S))]_{k }\|+2\epsilon
 \end{split}
 \end{equation*}
for any $\epsilon>0$, which proves  relation \eqref{qij}.
Now, it is easy to see that  $\varphi(S):=\oplus_{k=1}^\infty \Gamma_k(S)$ is
 an  isometric analytic operator   acting on $H^2(\DD)\otimes \ell^2$ such that
\begin{equation*}
\|[q_{ij}]_k\|_u= \|[q_{ij}(S\otimes I_{ \ell^2},\varphi(S))]_{k }\|
\end{equation*}
for any $[q_{ij}]_k\in M_k(\CC[z,w])$ and any $k\in \NN$.
The proof is complete.
\end{proof}
Note that since $S\otimes I_{ \ell^2}$ and $\varphi(S)$ are commuting  isometries, we can extend them to commuting unitaries, due to It\^ o's theorem \cite{SzFBK-book}, and using the spectral theorem we obtain
$$\|p(T_1,T_2)\|\leq \|p(S\otimes I_{ \ell^2},\varphi(S))\|= \|p\|_{\DD^2}.
$$
for any commuting contractions ${T}_1, {T}_2\in B(\cH)$ and  any   polynomial $p$ in $\CC[z,w]$, recovering in this way  And\^ o's inequality.
The closed non-self-adjoint algebra generated by
$ S \otimes I_{\ell^2} , \varphi({S})$ and the identity is denoted by
$\cA_u({\DD^2})$. In light of Theorem \ref{ando-bidisk},  $\cA_u({\DD^2})$ can be seen as the universal operator algebra for two commuting contractions.

We remark that the operator $\varphi(S)$ is Theorem \ref{ando-bidisk} is a direct sum of inner analytic operators on $H^2(\DD)\otimes \CC^k$, $k\in \NN$.
Consequently, using Agler-McCarthy's inequality \cite{AM} we deduce that there is a sequence of distinguished varieties $\{V_k\}_{k=1}^\infty$  in $\DD^2$ such that
$$\|p(T_1, T_2)\|\leq \sup_k\|p\|_{V_k}=\|p\|_{\DD^2}
$$
for any  commuting contractions $T_1$ and $T_2$, and any polynomial $p$ in two variables.

\bigskip

       %

\end{document}